\documentclass[Threeside,11pt]{article} 
		\usepackage{etoolbox}
		\usepackage{setspace}
		\usepackage{amsmath,amsthm,amsfonts,amssymb,amscd,mathrsfs}
		\usepackage{tikz}
		\usepackage{tikz-cd}
		\usepackage{float}
		\usepackage{subcaption}
		\usepackage[latin1]{inputenc}
		\usepackage[english]{babel}
		\usepackage{algorithm}
		\usepackage{adjustbox}
		\usepackage{cite}
		\usepackage{graphicx}
		\usepackage{colortbl,dcolumn}
		\usepackage{marvosym}
		\usepackage{ifsym}
		\usepackage{paralist}
		\usepackage{xcolor}
		\usepackage{array}
		\usepackage{geometry}
		\usepackage{enumitem}
	\DeclareFontFamily{U}{mathx}{\hyphenchar\font45}
	\DeclareFontShape{U}{mathx}{m}{n}{
		<5> <6> <7> <8> <9> <10>
		<10.95> <12> <14.4> <17.28> <20.74> <24.88>
		mathx10
	}{}
	\DeclareSymbolFont{mathx}{U}{mathx}{m}{n}
	\DeclareMathAccent{\widecheck}{0}{mathx}{"71}
		
		\definecolor{myPurple}{RGB}{109,90,207}
		\definecolor{myGreen}{RGB}{4,130,67}
		\definecolor{myGold}{RGB}{253,177,71}
		\definecolor{myBurgundy}{RGB}{63,1,44}
		\definecolor{myTeal}{RGB}{8,119,127}
		\definecolor{myCream}{RGB}{255,216,177}
		\definecolor{myOrange}{RGB}{225,119,1}
		\definecolor{mydeepgreen}{RGB}{3,100,50}

		\usetikzlibrary{shapes.geometric,arrows.meta,calc,decorations.markings,bending,arrows}
		\tikzset{
			ball3Donly/.style={
				circle, minimum size=1.4cm,
				shading=ball,
				ball color=myPurple,
				draw=myPurple
			}
		}
		\pgfmathsetmacro{\wNear}{2.0pt}
		\pgfmathsetmacro{\wMid}{1.3pt}
		\pgfmathsetmacro{\wFar}{0.7pt}
		\pgfmathsetmacro{\bend}{6mm}

		\makeatletter
		\patchcmd{\thebibliography}
		{\list}
		{\small\setstretch{1.2}\list}
		{}{}
		\makeatother

		\allowdisplaybreaks
		\newtheorem{lemma}{\bf Lemma}[section]
		\newtheorem{theorem}{\bf Theorem}[section]

		\newtheorem{remark}{\bf Remark}[section]
		
		\numberwithin{equation}{section}

		\usepackage[
		colorlinks,
		citecolor=blue,
		linkcolor=blue]{hyperref}

\begin{document}
		\title{{\sl Laskar's frequency map analysis revisited
		}}
		\author{Zhicheng Tong\thanks{School of Mathematics, Jilin University, Changchun 130012, P. R.  China. Email: \url{tongzc25@jlu.edu.cn}} 
			\and 
			Yong Li\thanks{Corresponding author. School of Mathematics, Jilin University, Changchun 130012, P. R.  China;   Center for Mathematics and Interdisciplinary Sciences,  Northeast  Normal  University, Changchun 130024, P. R. China. Email: \url{liyong@jlu.edu.cn}}
		}
		
		\date{}
		\maketitle
		
		\begin{abstract}
In this paper, we establish the exponential convergence of Laskar's pioneering frequency map analysis, strictly improving upon classical polynomial bounds. By utilizing appropriately chosen weighting functions, we achieve such exponential rates in the analytic quasi-periodic regime for frequency vectors satisfying Diophantine or Brjuno nonresonance conditions. Furthermore, we extend the theoretical framework beyond the analytic quasi-periodic regime into the analytic almost periodic setting. By employing and generalizing Bourgain's framework, we accommodate a much broader class of anisotropic spatial structures. These results yield the first unified theory of exponential convergence for frequency map analysis, revealing the interplay among analyticity, spatial structures, nonresonance conditions, the choice of weighting functions, and convergence rates. This framework has important implications for applications in fields including celestial mechanics. Moreover, the novel techniques developed herein yield new insights into the study of weighted Birkhoff averages.
		\vspace{2mm}
		\end{abstract}
		
		\tableofcontents
		
		\section{Introduction and main results}
		\setcounter{footnote}{0}
		\renewcommand{\thefootnote}{{\arabic{footnote}}}
The celebrated frequency map analysis, initiated by Laskar around 1990, was originally introduced to demonstrate and understand the chaotic behavior of the Solar System \cite{Las88,Las90}; we refer to \cite{Las99} for detailed theoretical explanations. This methodology subsequently expanded rapidly beyond its initial scope, finding extensive applications in celestial mechanics, galactic dynamics, atomic physics, and Hamiltonian toy models, among other fields. It has also become a standard diagnostic tool for beam dynamics in various particle accelerators, including, notably, hadron colliders, synchrotron light sources, high-intensity rings, and linear collider damping rings. For a comprehensive overview of these developments, we refer to \cite{Las88,Las90,LFC92,DL93,Las93a,Las93b,LR93,LF96,PL96,PL98,RSLN00,RL01,VW01,CWU03,Las03a,Las03b,CFL04,LCG+04,LRJ+04,CFL06,GLS09,LDV09,VDLC09,RRE12,Pap14,PJV+16,FL19,HMR20,LHH+20,BCG22,BGMT23,ML23,MAB+25,PR26} and the references therein. More recently, the fundamental ideas underlying Laskar's frequency map analysis have inspired the widespread application and theoretical development of weighted Birkhoff averages. For these recent advancements, we direct the reader to \cite{DSSY17,DY18,SM20,MS21,DM23,CCGd24,RB24,TL24a,TL24b,BdJC25,BHS25,MS25,SM25a,SM25b,TL25a,TL25b,BHS26,BBC26,GZE26,HXZ26,RKB26,Ton26,TL26a,TL26b,ZZL+26} and the references therein.

The primary motivation of this paper is to revisit Laskar's frequency map analysis. We demonstrate that under analytic assumptions, by choosing suitable weighting functions, the original polynomial rate of convergence can be improved to an exponential rate, provided the frequency vector satisfies either a Diophantine or a Brjuno nonresonance condition (Theorems \ref{TH2} and \ref{THBRUNO}). Furthermore, we extend the established theoretical framework from the quasi-periodic case to the almost periodic regime by employing Bourgain's framework and generalizing it to a broader lattice setting (Theorems \ref{TH3} and \ref{TH4}). These exponential convergence results are of considerable relevance to the applications of frequency map analysis, allowing for enhanced computational precision in such contexts. The present paper is devoted primarily to theoretical aspects, without delving into numerical simulations or practical applications; the interested reader is referred to the aforementioned works. The main contributions of this paper are detailed in Section \ref{SEC14}, and the proofs of our main results are presented in Section \ref{SEC2}.

\subsection{Frequency map analysis: Polynomial convergence}

To proceed directly to the main subject of Laskar's frequency map analysis, we bypass the introduction and description of Hamiltonian systems, non-degeneracy,  frequency maps, and KAM (Kolmogorov--Arnold--Moser) theory provided in \cite[Section 2]{Las99}. Instead, we directly adopt the setting introduced in \cite[Section 3]{Las99}, with minor notational adjustments. 

Consider a KAM quasi-periodic solution of a Hamiltonian system in complex form, given by
\begin{equation}\label{QP}
	f(t)=e^{i\nu_{1}t}+\sum_{k\in\mathbb{Z}^{d}\setminus \{(1,0,\dots,0)\} } a_{k}e^{i\langle k,\nu \rangle t}, \quad a_{k}\in\mathbb{C}, \quad t \in \mathbb{R}.
\end{equation}
Here, the frequency vector $\nu\in\mathbb{R}^d$ is assumed to be nonresonant, satisfying the Diophantine condition for some constants $\kappa_{\epsilon}>0$ and $\tau\geqslant d-1$:
\begin{equation}\label{DIO}
	|\langle k,\nu \rangle|>\frac{\kappa_{\epsilon}}{\|k\|_{\ell^1}^{\tau}},\quad \forall k \in \mathbb{Z}^d \setminus \{{\bf{0}}\},\quad \|k\|_{\ell^1}:=\sum_{j=1}^{d} |k_j|.
\end{equation}
The set of frequency vectors satisfying this condition for some $\kappa_{\epsilon}>0$ and fixed $\tau>d-1$ has full Lebesgue measure in $\mathbb{R}^d$; for detailed measure estimates, we refer the reader to \cite{Pos01}. The frequency analysis algorithm\footnote{Also known as the numerical analysis of the fundamental frequency (NAFF); see \cite[Section 3]{Las99} for details.} provides an approximation $f'(t)=\sum_{k=1}^{N} a_{k}' e^{i \omega_k' t}$ of $f$ from its numerical knowledge over a finite time span $[-T,T]$. The frequencies $\omega_k'$ and complex amplitudes $a_{k}'$ can be determined via an iterative scheme.

According to Laskar's frequency map analysis \cite{Las99}, the first frequency $\omega_1'$ is determined as the value of $\sigma$ that maximizes
\[\phi(\sigma):=\left|\langle f(t), e^{i\sigma t} \rangle _T ^{\chi} \right|\]
in a neighborhood of $\nu_1$; in other words, $\omega_1':= \arg\max_{\sigma \approx \nu_1} \phi (\sigma)$, which is also denoted by $\nu_1^T$ throughout this paper. Here, the weighted scalar product $\langle f(t),g(t) \rangle_T ^{\chi}$ is defined by
\[\langle f(t),g(t) \rangle _T ^{\chi}:=\frac{1}{2T}\int_{-T}^{T} \chi(t/T)f(t)\overline{g}(t)dt,\]
where $\chi$ is a weighting function---that is, a non-negative, even, and $C^\infty$ function on $[-1,1]$ such that\footnote{We note that the normalization employed here is distinct from that of weighted Birkhoff averages, albeit fundamental in the latter case.}
\[\frac{1}{2}\int_{-1}^{1}\chi(t)dt=1.\]
Given a weighting function $\chi$, we define the transform $\varphi_{\chi}$ of $\chi$ as
\[\varphi_{\chi}(x):=\langle e^{ixt},1 \rangle_{1}^{\chi},\]
which is clearly $C^\infty$ on $\mathbb{R}$. After identifying the initial periodic component $e^{i\omega'_1 t}$, we determine its corresponding complex amplitude $a'_1$ via orthogonal projection. Subsequently, the iterative procedure is applied to the residual function, which is given by $f_1(t)=f(t)-a'_1 e^{i\omega'_1 t}$. In the course of successive projections of $f$ onto the basis elements $e^{i\omega_k' t}$, it is necessary to orthogonalize the family of functions $\{e^{i\omega'_k t}\}_{k\in\mathbb{N}^+}$. Note that the quantity denoted by $\omega_k'$ also corresponds to $\nu_k^T$ for $k\in\mathbb N^+$, in accordance with \cite{Las99} and for the reader's convenience. Since frequency map analysis is a widely used and powerful numerical tool, establishing its convergence is of paramount importance. Specifically, for a sufficiently large integration time $T$, it is essential to estimate the error between the approximate frequency $\nu_1^T$ and the exact frequency $\nu_1$. Recall that the computation of $\nu_1^T$ depends on the weighting function $\chi$; hence the error depends intrinsically on the choice of $\chi$.

Theorem \ref{TH1} below, which combines the results of \cite[Theorem 1 \& Proposition 1]{Las99}, theoretically provides a polynomial-type error estimate for Laskar's frequency map analysis.
  
\begin{theorem}[Laskar 1999]
	\label{TH1}
 Let $\chi$ be a weighting function, and $\varphi = \varphi_{\chi}$ its transform with the asymptotic expressions as $x \to + \infty$
\[
\varphi(x) = \frac{g_0(x)}{x^n} + o\left(\frac{1}{x^n}\right), \quad D\varphi(x) = \frac{g_1(x)}{x^n} + o\left(\frac{1}{x^n}\right), \quad D^2\varphi(x) = \frac{g_2(x)}{x^n} + o\left(\frac{1}{x^n}\right),  
\]
where $n \geqslant 1$, and $g_0 , g_1 $ and $g_2 $ are bounded on $\mathbb{R}$. Let $f $ be a quasi-periodic function of the form \eqref{QP}, and for all $k$, $\Omega_k := \langle k, \nu \rangle - \nu_1$; assume that $\sum\nolimits_{k \in {\mathbb{Z}^d}} {| {{a_k}\Omega _k^{ - m}} |} $ is convergent for $m = 0, 1$, and $n$. Then in frequency map analysis, as $T \to +\infty$, the error between the approximate frequency $\nu_1^{T}$ and the exact frequency $\nu_1$ satisfies
\[\nu_1^T-\nu_1    = \frac{1}{\varphi''(0)T^{n+1}} \sum_{k\in\mathbb{Z}^{d}\setminus \{(1,0,\dots,0)\} } \frac{\Re(a_k)}{\Omega_k^n} g_1(\Omega_k T) + o\left(\frac{1}{T^{n+1}}\right) 
=\mathcal{O}\left(\frac{1}{T^{n+1}}\right). \]
Consequently, for the cosine-like weighting function 
\[\chi_{m}(t)=\frac{2^{m}(m!)^{2}}{(2m)!}(1+\cos(\pi t))^{m},\quad m \in \mathbb{N}^+,\]
 we have
\[\nu_{1}^{T}-\nu_{1}=\frac{(-1)^{m}\pi^{2m}(m!)^{2}}{A_{m}T^{2m+2}}\sum_{k\in\mathbb{Z}^{d}\setminus \{(1,0,\dots,0)\} }\frac{\Re(a_{k})}{\Omega_{k}^{2m+1}}\cos(\Omega_{k}T)+o\left(\frac{1}{T^{2m+2}}\right)=\mathcal{O}\left(\frac{1}{T^{2m+2}}\right),\]
with
\[A_{m}=-\frac{2}{\pi^{2}}\left(\frac{\pi^{2}}{6}-\sum_{k=1}^{m}\frac{1}{k^{2}}\right)<0.\]
In particular, for the Hanning window filter $ \chi_{1}(t)=1+\cos(\pi t) $, the error becomes $ \mathcal{O}\left(T^{-4}\right) $.
\end{theorem}
\begin{remark}
Given the assumption that the frequency vector $\nu$ satisfies the Diophantine condition \eqref{DIO}, it is evident that if $f$ in \eqref{QP} is analytic, $\sum\nolimits_{k \in {\mathbb{Z}^d}} {| {{a_k}\Omega _k^{ - m}} |} $ converges for $m = 0, 1$, and $ n $, thereby establishing Theorem \ref{TH1}.
\end{remark}
\begin{remark}
Without a non-trivial weighting function (i.e., using the trivial weighting function $\chi_0(t)$), the ordinary fast Fourier transform (FFT) method yields an error of $\mathcal{O}\left(T^{-1}\right)$, which is significantly slower than the rates established in Theorem \ref{TH1}.
\end{remark}

\subsection{From polynomial to exponential convergence}\label{SUBSEC12}

In numerical simulations, Laskar frequently employed the Hanning window filter $ \chi_1 $, which yields a convergence rate of $\mathcal{O}(T^{-4})$ for his frequency map analysis (see Theorem \ref{TH1}). Furthermore, in \cite[Remark 2]{Las99}, he introduced a specific exponential weighting function, henceforth referred to as Laskar's weighting function, defined by 
	\begin{equation}\notag
	w_{\rm Las}(x) := 
	\begin{cases}
	c_1\exp \left( { - {{\left( {1 - x^{2}} \right)^{-1}}}} \right), & x \in (-1,1), \\
		0, & x \in \{-1, 1\},
	\end{cases}\quad  
	\end{equation}
  where $c_1 > 0$ is a normalization constant chosen such that $\frac{1}{2} \int_{-1}^1 w_{\mathrm{Las}}(t) dt = 1$. Remarkably, this function guarantees\footnote{While Laskar did not provide a rigorous convergence analysis of this weighting function for weighted Birkhoff averages---a result later established in \cite{DSSY17,DY18}---its rigorous justification for frequency map analysis readily follows from Theorem \ref{TH1}. Moreover, we note that the analytical frameworks of both approaches share certain similarities, particularly in the context of continuous-time flows independently studied in \cite{DM23,TL24a}.} a convergence rate of $\mathcal{O}(T^{-\upsilon})$ for frequency map analysis, where $\upsilon $ is an arbitrary positive integer. It was noted in \cite{DSSY17} that ``There seems to be no advantage to using his lower order methods\footnote{This refers to the application of the Hanning window filter $ \chi_1 $.} than $C^\infty$ filter.'' In practice, employing Laskar's weighting function $w_{\rm Las}$ yields an exponential rate of convergence that is substantially faster than that of the Hanning window filter $\chi_1$; consequently, it has been widely adopted in recent studies on weighted Birkhoff averages. Subsequently, Laskar's weighting function $w_{\rm Las}$ was generalized to a parameterized variant $w_{\rm Las}^{p,q}$ with $ p,q>0 $, defined as
 	\begin{equation}\notag 
 	w_{\rm Las}^{p,q}(x) := 
 	\begin{cases}
 		c_2\exp \left( { - {{\left( {1 + x} \right)}^{ - p}}{{\left( {1 - x} \right)}^{ - q}}} \right), & x \in (-1,1), \\
 		0, & x \in \{-1, 1\},
 	\end{cases}\quad
 \end{equation}
 where $c_2 > 0$ is chosen such that $\frac{1}{2} \int_{-1}^1 w_{\mathrm{Las}}^{p,q}(t) dt = 1$. This variant exhibits exceptional acceleration properties for weighted Birkhoff averages (see, e.g., \cite{DSSY17,CCGd24,TL25b,Ton26,TL26a}).
 
 To date, however, the convergence of frequency map analysis based on Laskar's weighting function $w_{\mathrm{Las}}$ has remained restricted to the polynomial rate originally established by him. As a principal contribution of the present paper, we adapt certain ideas from \cite{TL24a,TL25b,TL26a} to establish, in Theorem \ref{TH2}, the first exponential convergence result for frequency map analysis utilizing the generalized weighting function $w_{\mathrm{Las}}^{p,q}$ under the assumptions of analyticity and the Diophantine condition. Here, the function $f$ in \eqref{QP} is said to be analytic if there exist constants $C_f, r > 0$ such that $|a_k| \leqslant C_f \exp(-r\|k\|_{\ell^1})$ for all $k \in \mathbb{Z}^d$; see \cite[Lemma 1]{Sal04} for details.

 
\begin{theorem}[\textbf{Main Theorem I}]
\label{TH2}
Let $w_{\rm Las}^{p,q}$, with $p,q>0$, be the parameterized variant of Laskar's weighting function, and let $f$ be an analytic quasi-periodic function of the form \eqref{QP} whose frequency vector $\nu$ satisfies the Diophantine condition \eqref{DIO}. Then in frequency map analysis, as $T \to +\infty$, the error between the approximate frequency $\nu_1^{T}$ and the exact frequency $\nu_1$ satisfies
\[\nu_1^T - \nu_1 = \mathcal{O}\big(e^{-c_{\rm I} T^\zeta}\big),\]
where $c_{\rm I}>0$ is a constant, and $\zeta := (\tau + \beta)^{-1}$ with $\beta := 1 + (\min\{p,q\})^{-1}$.
\end{theorem}

Within the framework of KAM theory, the discovery that arithmetic properties weaker than the classical Diophantine condition \eqref{DIO} can still guarantee the persistence of quasi-periodic solutions has motivated extensive research. The most well-known criterion is the Brjuno (Bruno) condition, given by
\begin{equation}\label{BRU}
	|\langle k,\nu \rangle| > \frac{\kappa_{\epsilon}}{\Delta(\|k\|_{\ell^1})},\quad \forall k \in \mathbb{Z}^d \setminus \{\mathbf{0}\},
\end{equation}
where $ \kappa_{\epsilon}>0 $ is a constant, and $\Delta : [1,+\infty) \to [1,+\infty)$ is a monotonically increasing approximation function satisfying
\begin{equation}\label{BRU2}
	\int_1^{+\infty} \frac{\log \Delta(x)}{x^2} dx < +\infty,\quad \frac{\log \Delta(x)}{x} \searrow 0 \text{\;\;as\;\;} x \to +\infty.
\end{equation}
In this regard, we refer the reader to \cite{Pos90,Pos11,Rus10,Bou19} for further details, and to \cite{BF19,BF21} for more general conditions involving ultra-differentiable regularity.

\begin{theorem}[\textbf{Main Theorem II}]\label{THBRUNO}
Let $w_{\rm Las}^{p,q}$, with $p,q>0$, be the parameterized variant of Laskar's weighting function, and let $f$ be an analytic quasi-periodic function of the form \eqref{QP} whose frequency vector $\nu$ satisfies the Brjuno condition \eqref{BRU}--\eqref{BRU2}. Then in frequency map analysis, as $T \to +\infty$, the error between the approximate frequency $\nu_1^{T}$ and the exact frequency $\nu_1$ satisfies
	\[\nu_1^T - \nu_1 = \mathcal{O}\Big(\exp\big(-c_{\mathrm{II}} \log T \log\log T\big)\Big),\]
	where $c_{\mathrm{II}}>0$ is a constant that can be chosen arbitrarily large.
\end{theorem}

\subsection{Towards almost periodicity}
In \cite[Remark 4]{Las99}, Laskar remarked that his polynomial convergence results (see Theorem \ref{TH1} of the present paper) remain applicable to more general almost periodic functions of the form $ f(t) = {e^{i{\omega _1}t}} + \sum\nolimits_{k \in \mathbb{Z}} {{a_k}{e^{i{\omega _k}t}}}  $ with $ {\Omega _k} = {\omega _k} - {\omega _1} \ne 0 $, provided that the same hypotheses hold. This assertion is indeed valid, as his argument does not intrinsically rely upon the finite-dimensional Diophantine condition in \eqref{DIO}. Nevertheless, achieving exponential convergence necessitates more delicate spatial structures and refined nonresonance conditions. While various frameworks are available, we shall first adopt Bourgain's setting in Section \ref{SUBSEC131}, and subsequently explore a more general lattice setting in Section \ref{SUBSEC132}. It is worth noting that these frameworks have been extensively investigated within the scope of KAM theory; therefore, the results established in the present section can be directly combined with applicable results from existing KAM theory to deduce the exponential convergence of frequency map analysis.

\subsubsection{Bourgain's setting}\label{SUBSEC131}
The framework considered here was originally introduced by Bourgain \cite{Bou05} and subsequently refined in \cite{MP21} for the construction of KAM almost periodic solutions.

Fix an integer $\eta \geqslant 2$. We define the lattice of infinite integer sequences with finite support as
\[\mathbb{Z}_*^\mathbb{N} := \left\{ k \in \mathbb{Z}^\mathbb{N} :\; |k|_\eta := \sum_{j \in \mathbb{N}} \langle j \rangle^\eta |k_j| < +\infty, \; \langle j \rangle := \max\{1, j\} \right\}.\]
By definition, any sequence $k \in \mathbb{Z}_*^\mathbb{N}$ possesses only finitely many non-zero components. The weighted metric $|\cdot|_\eta$ imposes a crucial summability condition, which will play a pivotal role in our subsequent analysis.

In the sequel, we investigate a framework more general than that considered in \cite{Las99}. Let
\begin{equation}\label{AP}
	f(t) = e^{i\nu_1 t} + \sum_{k \in \mathbb{Z}_*^\mathbb{N} \setminus \{( 1, 0, \dots)\}} a_k e^{i\langle k, \nu \rangle t}, \quad a_k \in \mathbb{C}, \quad t \in \mathbb{R}
\end{equation}
be a KAM almost periodic function associated with a Hamiltonian lattice or a Hamiltonian PDE. We assume that the frequency vector $\nu$ satisfies the infinite-dimensional Diophantine condition for some constants $\gamma > 0$ and $\mu > 1$:
\begin{equation}\label{DIO2}
	|\langle k, \nu \rangle| > \frac{\gamma}{\prod\limits_{j \in \mathbb{N}} (1 + \langle j \rangle^\mu |k_j|^\mu)}, \quad \forall k \in \mathbb{Z}^\mathbb{N} \text{\;\;with\;\;} 0 < \sum\limits_{j \in \mathbb{N}} {|{k_j}|} < + \infty .
\end{equation}
It can be shown that for any interval $[\delta_1, \delta_2]$ with $ 0<\delta_1<\delta_2 $, the set of frequency vectors satisfying \eqref{DIO2} for some $ \gamma>0 $ has full measure in the Cartesian product space $[\delta_1, \delta_2]^\mathbb{N}$ with respect to the normalized product probability measure. We refer the reader to \cite{Bou05} for an analogous measure-theoretic argument. Furthermore, the function $f$ given by \eqref{AP} is said to be analytic if there exist constants $C_f, r > 0$ such that the Fourier coefficients satisfy the exponential decay estimate $|a_k| \leqslant C_f \exp(-r|k|_\eta)$ for all $k \in \mathbb{Z}_*^\mathbb{N}$; see, e.g., \cite{MP21}.

With the preceding preparations, we are now ready to state our first exponential convergence result for the almost periodic frequency map analysis.

\begin{theorem}[\textbf{Main Theorem III}]\label{TH3}
Let $w_{\rm Las}^{p,q}$, with $p,q>0$, be the parameterized variant of Laskar's weighting function, and let $f$ be an analytic almost periodic function of the form \eqref{AP} whose frequency vector $\nu$ satisfies the infinite-dimensional Diophantine condition \eqref{DIO2}. Then in frequency map analysis, for any $\rho < 1+\eta$ (with integer $\eta \geqslant 2$ defined previously), as $T \to +\infty$, the error between the approximate frequency $\nu_1^{T}$ and the exact frequency $\nu_1$ satisfies
	\[\nu_1^T - \nu_1 =  \mathcal{O}\Big(\exp\big(-c_{\mathrm{III}} (\log T)^\rho\big)\Big),\]
	where $c_{\mathrm{III}}>0$ is a constant.
\end{theorem}

\subsubsection{A general lattice setting}\label{SUBSEC132}

To obtain almost periodic solutions via KAM theory, various frameworks distinct from Bourgain's setting have been extensively investigated; see, for instance, \cite{Pos90,TL23,CGP24,DFP26}. To quantitatively establish the exponential convergence of frequency map analysis, we proceed in the spirit of \cite{TL23,CGP24} to generalize Bourgain's spatial structure, while still utilizing the infinite-dimensional Diophantine condition \eqref{DIO2} for convenience. 

Given a monotonically increasing function $\Phi \colon [1,+\infty) \to [1,+\infty)$, 
 we define as in Section \ref{SUBSEC131} the lattice of infinite integer sequences with finite support:
\[
\mathbb{Z}_{\Phi}^\mathbb{N} := \left\{ k \in \mathbb{Z}^\mathbb{N} :\; \|k\|_\Phi := \sum_{j \in \mathbb{N}} \Phi (\langle j \rangle) |k_j| < +\infty, \; \langle j \rangle := \max\{1, j\} \right\}.
\]
In particular, if we let $\Phi(x)=x^\eta$ for some integer $\eta \geqslant 2$, then $\mathbb{Z}_{\Phi}^\mathbb{N}$ corresponds to Bourgain's spatial structure $\mathbb{Z}_{*}^\mathbb{N}$. This form of spatial structure intrinsically exhibits a certain anisotropy, which is crucial for handling small divisors in both KAM theory and frequency map analysis. As observed in \cite[Section 6.2]{TL23}, the newly introduced spatial structure exhibits intriguing connections with the infinite-dimensional Diophantine condition \eqref{DIO2}, which will be further explored in the analysis of Section \ref{SUBSEC24}. 

Based on the spatial structure $\mathbb{Z}_{\Phi}^\mathbb{N}$, let us consider a KAM almost periodic function similar to that in \eqref{AP}:
\begin{equation}\label{AP2}
f(t) = e^{i\nu_1 t} + \sum_{k \in \mathbb{Z}_{\Phi}^\mathbb{N} \setminus \{( 1, 0, \dots)\}} a_k e^{i\langle k, \nu \rangle t}, \quad a_k \in \mathbb{C}, \quad t \in \mathbb{R},	
\end{equation}
where the frequency vector $\nu$ satisfies the infinite-dimensional Diophantine condition \eqref{DIO2}. We remark that other types of nonresonance conditions may also be considered, though they inevitably necessitate the establishment of corresponding small divisor estimates; we refer the reader to \cite{TL23}.  Furthermore, in this setting, we say that $f$ is analytic if there exists $r>0$ such that the following weighted norm is bounded:
	\[\sum_{k \in \mathbb{Z}_\Phi^\mathbb{N}} |a_k| \exp(r \|k\|_\Phi) < +\infty.\]
It is worth noting that this differs from the analyticity defined in Sections \ref{SUBSEC12} and \ref{SUBSEC131}, where we directly required the exponential decay of the Fourier coefficients. However, the introduction of a general spatial structure herein introduces essential difficulties (such as cardinality estimates and preventing the blow-up of the Fourier series); therefore, we opt to impose a slightly stronger notion of analyticity.

Building upon the general spatial structure introduced in this section, we now state our second exponential convergence result for the almost periodic frequency map analysis.

\begin{theorem}[\textbf{Main Theorem IV}]\label{TH4}
Let $w_{\rm Las}^{p,q}$, with $p,q>0$, be the parameterized variant of Laskar's weighting function, and let $f$ be an analytic almost periodic function of the form \eqref{AP2} whose frequency vector $\nu$ satisfies the infinite-dimensional Diophantine condition \eqref{DIO2}. Then in frequency map analysis, as $T \to +\infty$, the error between the approximate frequency $\nu_1^{T}$ and the exact frequency $\nu_1$ satisfies
	\[ \nu_1^T - \nu_1 = \mathcal{O}\Big(\exp\big(-c_{\mathrm{IV}}\Gamma^{-1}(c_{\mathrm{V}}\log T)\big)\Big), \]
	where $c_{\mathrm{IV}}$ and $c_{\mathrm{V}}$ are positive constants, $\Gamma(x) := \Psi(x)\log\Lambda(x)$, $\Psi$ is the inverse function of the map $x \mapsto \int_0^x \Phi(y)dy$ (with $\Phi$ defined previously), and
	\[
	\Lambda(x) := \begin{cases}
		x, & \text{if\;\;} \Phi(x) \geqslant x, \\
		\Phi^{-1}(x), & \text{if\;\;} \Phi(x) < x.
	\end{cases}
	\]
	Consequently, we have the following estimates as $ T \to +\infty $:
	\begin{enumerate}[label=(\Roman*)]
		
		\item \label{item1} If $ \Phi(x) \sim \log x $ as $ x \to +\infty $, then there exists a constant $ c_{\mathrm{VI}} > 0 $ such that
		\[ \nu_1^T - \nu_1 = \mathcal{O}\left(\exp\left(-c_{\mathrm{VI}}\sqrt{\log T\log\log T}\right)\right). \]
		
		\item \label{item2} If $ \Phi(x) \sim e^x $ as $ x \to +\infty $, then there exists a constant $ c_{\mathrm{VII}} > 0 $ such that
		\[ \nu_1^T - \nu_1 = \mathcal{O}\Big(\exp\big(-\exp(c_{\mathrm{VII}}\sqrt{\log T})\big)\Big). \]
		
		\item \label{item3} In particular, for any given function $ \mathscr{I}: [1,+\infty) \to [1,+\infty) $ satisfying $ \mathscr{I}(x) = o(x^\beta) $ as $ x \to +\infty $ with $\beta = 1 + (\min\{p,q\})^{-1}$, there always exists a suitable $ \Phi $ such that
		\[ \nu_1^T - \nu_1 = \mathcal{O}\Big(\exp\big(-\mathscr{I}(T)\big)\Big). \]
		
	\end{enumerate}
\end{theorem}

We conclude this section by emphasizing that the convergence rate for the almost periodic frequency map analysis established in Theorem \ref{TH4} does not necessarily exhibit exponential decay. Depending on the explicit structure of the space $\mathbb{Z}_{\Phi}^\mathbb{N}$, the decay may be slower than any polynomial decay\footnote{This statement is merely an intuitive description, since an upper bound estimate for the rate does not necessarily reflect the actual rate itself.} (see Case \ref{item1}), or of an exponential type (see Cases \ref{item2} and \ref{item3}). However, provided that the regularity of $f$ is sufficiently high (e.g., stronger than the analyticity considered here), an exponential decay rate is always achievable. For the sake of brevity, we do not pursue this direction further here.\footnote{This issue will be addressed quantitatively in a separate paper within the context of weighted Birkhoff averages.}

\subsection{Main contributions}\label{SEC14}
The primary contributions of this paper are summarized in four main results: Theorems \ref{TH2}, \ref{THBRUNO}, \ref{TH3}, and \ref{TH4}. By employing appropriate weighting functions (namely the parameterized variant of Laskar's weighting function $w_{\rm Las}^{p,q}$ with $ p,q>0 $), these results improve, for the first time, the convergence rate of Laskar's frequency map analysis (Theorem \ref{TH1}) from polynomial to exponential. This yields an improvement of substantial theoretical and practical significance; in particular, it upgrades the polynomial convergence previously observed in the literature (see \cite{Las88,Las90,LFC92,DL93,Las93a,Las93b,LR93,LF96,PL96,PL98,RSLN00,RL01,VW01,CWU03,Las03a,Las03b,CFL04,LCG+04,LRJ+04,CFL06,GLS09,LDV09,VDLC09,RRE12,Pap14,PJV+16,FL19,HMR20,LHH+20,BCG22,BGMT23,ML23,MAB+25,PR26} and the references therein) to an exponential rate. Moreover, the study of such general weighting functions aligns naturally with Laskar's original motivation (see \cite[Section 7.3]{Las03b}). It is worth noting that we do not claim our exponential convergence rate will always outperform the polynomial one numerically (for instance, when employing cosine-like weighting functions) over short time intervals, a phenomenon consistent with observations in \cite[Sections 5.1 \& 7.3]{Las03b}. This may be attributed to potentially large prefactors in the exponential bounds. Furthermore, such discrepancies could stem from computational errors in evaluating the integral $\langle f(t), e^{i\sigma t} \rangle_T^{\chi} $ during the frequency map analysis, inaccuracies in identifying the approximate frequency $\sigma$ that maximizes the modulus of this integral, or other related numerical artifacts. However, the asymptotic superiority of the exponential rate is guaranteed over sufficiently long time scales. For a more precise explanation of this phenomenon in the context of weighted Birkhoff averages, we refer the reader to \cite[Section 1.1.4]{TL26a}. 

Building upon the foundation of Laskar's frequency map analysis, Theorems \ref{TH2} and \ref{TH3} partially incorporate novel techniques recently developed by the authors concerning weighted Birkhoff averages \cite{TL24a,TL24b,TL25b,TL26a}. Specifically, Theorem \ref{TH2} explicitly reveals how the Diophantine nonresonance condition of the frequency vector affects the exponential convergence rate, whereas Theorem \ref{TH3} extends Laskar's quasi-periodic framework to encompass Bourgain's almost periodic setting.

Furthermore, Theorem \ref{THBRUNO} establishes a unified exponential convergence theory for frequency map analysis concerning Brjuno-type frequency vectors. This generalization poses a considerable challenge, as Brjuno nonresonance lacks the simple power-law lower bound of the Diophantine condition, being instead formulated via integrability and monotonicity criteria. 

Theorem \ref{TH4} further generalizes almost periodic frequency map analysis beyond Bourgain's setting, providing a framework that enables the investigation of physical systems characterized by more general anisotropic spatial structures. To build intuition, Theorem \ref{TH4} incorporates concrete examples (Cases \ref{item1} and \ref{item2}) and elucidates the intricate interplay among analyticity, spatial structures, nonresonance conditions, weighting functions, and convergence rates within the almost periodic regime (Case \ref{item3}).

Even when viewed strictly through the technical lens of weighted Birkhoff averages, Theorems \ref{THBRUNO} and \ref{TH4}--particularly the latter--introduce novel methodologies. Although these results indirectly imply the exponential convergence of weighted Birkhoff averages under a comparable theoretical framework (see, for instance, \cite{DSSY17,DY18,DM23,TL24a,TL24b,TL25a,TL25b,Ton26,TL26a,TL26b}), we refrain from formulating these implications as formal theorems to avoid overly cumbersome notation.

\section{Proofs of main results}\label{SEC2}
\subsection{Proof of Theorem \ref{TH2}: Exponential convergence for frequency map analysis via quasi-periodicity}\label{SEC21}
The analysis presented in this section is partially adapted from Laskar's original proof \cite{Las99}, albeit with certain refinements and modifications. The exponential convergence analysis in the latter half, however, is entirely distinct, relying primarily on certain novel ideas introduced in \cite{TL24a,TL24b,TL25b,TL26a}.

For brevity of notation, we shall henceforth write $w$ in place of the parameterized variant of Laskar's weighting function $w_{\rm Las}^{p,q}$ for $p,q>0$, and make corresponding substitutions for the quantities involving $\chi$ defined in the Introduction.

We begin with a preliminary lemma.

\begin{lemma}\label{lemma:1}
	Let $\varphi = \varphi_w$ denote the transform of the weighting function $w$. Then $D^n\varphi(x) \in \mathbb{R}$ for all $n \in \mathbb{N}$. Furthermore, we have $|\varphi(x)| \leqslant 1$, $\varphi(0) = 1$, $|D\varphi(x)| \leqslant 1$, $D\varphi(0) = 0$, $|D^2\varphi(x)| \leqslant 1$, and $D^2\varphi(0) < 0$.
\end{lemma}
\begin{proof}
	By definition,
	\[
	\varphi(x) = \varphi_w(x) = \langle e^{ixt}, 1 \rangle_1^w = \frac{1}{2}\int_{-1}^1 w(t)e^{ixt}dt.
	\]
	Since $w$ is an even function, $\varphi(x)$ is real-valued, and thus $D^n\varphi(x) \in \mathbb{R}$ for all $n \in \mathbb{N}$. Let $ \mathcal{A}:= \frac{1}{2}\int_{-1}^1 w(t)dt = 1$ for brevity. Then direct integration yields
	\[
	|\varphi(x)| \leqslant \mathcal{A}, \quad  \varphi(0) = \mathcal{A}.
	\]
	Differentiating under the integral sign gives the bounds
	\[
	|D\varphi(x)| \leqslant \frac{1}{2}\int_{-1}^1 w(t)|t|dt \leqslant \mathcal{A}, \quad |D^2\varphi(x)| \leqslant \frac{1}{2}\int_{-1}^1 w(t)t^2dt \leqslant \mathcal{A}.
	\]
	For the derivatives at the origin, the parity of the integrands implies
	\[
	D\varphi(0) = \frac{i}{2}\int_{-1}^1 w(t)tdt = 0, \quad D^2\varphi(0) = - \frac{1}{2}\int_{-1}^1 w(t)t^2dt < 0,
	\]
	which completes the proof.
\end{proof}

Let $f$ be the analytic quasi-periodic function defined in \eqref{QP} with the Diophantine frequency $\nu$ satisfying \eqref{DIO}. For a given $T > 0$, we define $\Xi := \mathbb{Z}^d \setminus \{(1,0,\dots,0)\}$ and set $\Omega_k := \langle k, \nu \rangle - \nu_1$ for all $k \in \Xi$. We then define the function
\[
\psi(x) := \varphi(x) + \sum_{k \in \Xi} a_k \varphi(x + \Omega_k T).
\]
Recall that from frequency map analysis, we have $\nu_1^T = \arg\max \phi(\sigma)$ near $\nu_1$. By means of the substitution $x = (\nu_1 - \sigma)T$, we obtain
\begin{align*}
	\phi(\sigma) &= \left| \langle f(t), e^{i\sigma t} \rangle_T^w \right| \\
	&= \left| \frac{1}{2T} \int_{-T}^T w(t/T) f(t) e^{-i\sigma t}  dt \right| \\
	&= \left| \frac{1}{2T} \int_{-T}^T w(t/T) \left( e^{i\nu_1 t} + \sum_{k \in \Xi} a_k e^{i\langle k,\nu \rangle t} \right) e^{-i\sigma t}  dt \right| \\
	&= \left| \frac{1}{2} \int_{-1}^1 w(s) e^{ixs}  ds + \sum_{k \in \Xi} a_k \left( \frac{1}{2} \int_{-1}^1 w(s) e^{i(x + \Omega_k T)s}  ds \right) \right| \\
	&= |\psi(x)|.
\end{align*}
This leads us to investigate the maximum modulus of $\psi(x)$, or equivalently, the maximum of its square $\psi(x)\overline{\psi}(x)$. Consequently, for the point $x$ at which this maximum is attained, one necessarily has
\[
D\left( \psi(x)\overline{\psi}(x) \right) = D\psi(x) \overline{\psi}(x) + \psi(x) D\overline{\psi}(x) = 0.
\]
Thus, by Lemma \ref{lemma:1}, a straightforward computation yields
\begin{align}
	F(x,T) :&= \frac{1}{2}\left( D\psi(x)\overline{\psi}(x) + \psi(x)D\overline{\psi}(x) \right)\notag \\
	&= \frac{1}{2} \left( D\varphi(x) + \sum_{k \in \Xi} a_k D\varphi(x + \Omega_k T) \right) \left( \varphi(x) + \sum_{k \in \Xi} \overline{a}_k \varphi(x + \Omega_k T) \right)\notag \\
	&\quad + \frac{1}{2} \left( \varphi(x) + \sum_{k \in \Xi} a_k \varphi(x + \Omega_k T) \right) \left( D\varphi(x) + \sum_{k \in \Xi} \overline{a}_k D\varphi(x + \Omega_k T) \right)\notag \\
	&= D\varphi(x)\varphi(x) + \sum_{k \in \Xi} \Re(a_k) \left( \varphi(x) D\varphi(x + \Omega_k T) + D\varphi(x) \varphi(x + \Omega_k T) \right) \notag \\
	&\quad + \sum_{k,l \in \Xi} \Re(a_k \overline{a}_l) \varphi(x + \Omega_k T) D\varphi(x + \Omega_l T) = 0.\label{def:F}
\end{align}

The following lemma establishes the solvability of $x=x(T)$ in terms of $T$.

\begin{lemma}\label{lemma:2}
	There exists a neighborhood $\mathscr{U}$ of $+\infty$ and a unique continuous map $x=x(T)$ such that
	\[
	\lim_{T \to +\infty} x(T) = 0 \quad \text{and} \quad F(x(T),T) = 0, \quad \forall T \in \mathscr{U}.
	\]
\end{lemma}
\begin{proof}
	Let $A>0$ be given, and consider $x \in [-A,A]$. We first assert that 	
	\[
	\lim_{T \to +\infty} F(x,T) = \varphi(x)D\varphi(x)
	\]
	holds uniformly in $x$, or equivalently, that the remaining terms in \eqref{def:F} vanish as $T \to +\infty$. Once this is established, $F(x, T)$ can be continuously extended to $[-A,A] \times (0,+\infty]$. Note that by the analyticity of $f$, we have
	\[
	\sum_{k \in \mathbb{Z}^n} |\Re(a_k)| \leqslant \sum_{k \in \mathbb{Z}^n} |a_k| < +\infty, \quad \sum_{k,l \in \mathbb{Z}^n} |\Re(a_k \overline{a}_l)| \leqslant \sum_{k,l \in \mathbb{Z}^n} |a_k| |a_l| = \left( \sum_{k \in \mathbb{Z}^n} |a_k| \right)^2 < +\infty,
	\]
	while the Riemann--Lebesgue lemma implies that for any fixed $x \in [-A,A]$, $j \in \mathbb{N}$, and $k \in \Xi$,
	\[
	\lim_{T \to +\infty} D^j\varphi(x + \Omega_k T) = 0.
	\]
	Therefore, the assertion follows from the Lebesgue dominated convergence theorem.\footnote{A delicate estimate via Taylor expansion as in \cite{Las99} is not required here, though it is certainly feasible.}
	
	Furthermore, by \eqref{def:F}, we derive
	\begin{align}
		\frac{\partial F}{\partial x}(x,T) &= D^2\varphi(x) \varphi(x) + (D\varphi(x))^2 \notag \\
		&\quad + \sum_{k \in \Xi} \Re(a_k) \left( 2D\varphi(x)D\varphi(x + \Omega_k T) + \varphi(x)D^2\varphi(x + \Omega_k T) + D^2\varphi(x)\varphi(x + \Omega_k T) \right) \notag \\
		&\quad + \sum_{k,l \in \Xi} \Re(a_k \overline{a}_l) \left( D\varphi(x + \Omega_k T)D\varphi(x + \Omega_l T) + \varphi(x + \Omega_k T)D^2\varphi(x + \Omega_l T) \right). \notag 
	\end{align}
	Following a similar argument as above, we deduce that for any fixed $x$,
	\begin{equation}\label{def:Fdao}
		\lim_{T \to +\infty} \frac{\partial F}{\partial x}(x,T) = D^2\varphi(x) \varphi(x) + (D\varphi(x))^2.
	\end{equation}
	Hence, the continuous extension, still denoted by $F(x,T)$, admits a continuous first derivative $\frac{\partial F}{\partial x}(x,T)$ on $[-A,A] \times (0,+\infty]$. 
	
	Now, by \eqref{def:Fdao} and Lemma \ref{lemma:1}, we obtain    
	\[
	\frac{\partial F}{\partial x}(0, +\infty) = D^2\varphi(0)\varphi(0) + (D\varphi(0))^2 = D^2\varphi(0) < 0.
	\]
	Thus, by applying the implicit function theorem \cite[p.~176]{Sch92} to $F(x,T)$ at the point $(0,+\infty)$, we conclude the proof.
\end{proof}

Given the continuous map $x(T)$ obtained in Lemma \ref{lemma:2}, we now investigate the asymptotic behavior of the series $\sum_{k \in \Xi} a_k D^j \varphi(x(T) + \Omega_k T)$ as $T \to +\infty$ for $0\leqslant j\leqslant  1$. We shall demonstrate that these series decay exponentially. To this end,  we define $\zeta := (\tau + \beta)^{-1}$ with $\beta := 1 + (\min\{p,q\})^{-1}$ (recalling that $p,q>0$ are the constants in the weighting function $w$), and decompose the index set $\Xi$ into two disjoint parts for sufficiently large $T$:
\[
\mathcal{S}_{\mathrm{T}} := \left\{ k \in \Xi : \|k\|_{\ell^1} \leqslant T^\zeta \right\}, \quad \mathcal{S}_{\mathrm{R}} := \left\{ k \in \Xi : \|k\|_{\ell^1} > T^\zeta \right\},
\]
where $\|k\|_{\ell^1} := \sum_{i=1}^d |k_i|$. Consequently, these series can be split as follows:
\begin{equation}\label{fenjie}
	\sum_{k \in \Xi} a_k D^j\varphi(x(T) + \Omega_k T) = \sum_{k \in \mathcal{S}_{\mathrm{T}}} a_k D^j \varphi(x(T) + \Omega_k T) + \sum_{k \in \mathcal{S}_{\mathrm{R}}} a_k D^j \varphi(x(T) + \Omega_k T), \quad 0\leqslant j\leqslant  1.
\end{equation}
As will be seen below, the truncated principal part and the truncated remainder require fundamentally different treatments. The former presents the primary difficulty, as it necessitates overcoming the small divisor problem via the Diophantine nonresonance condition. The latter, by contrast, relies essentially on regularity and is comparatively straightforward. We now proceed to estimate these two terms using distinct methods.

\begin{lemma}\label{lemma:3}
	For $\zeta$ as defined above, there exists a constant $c_1 > 0$ such that the following asymptotic estimate holds as $T \to +\infty$:
	\[\sum_{k \in \mathcal{S}_{\mathrm{T}}} a_k D^j\varphi(x(T) + \Omega_k T) = \mathcal{O}\big(\exp(-c_1 T^\zeta)\big),\quad 0 \leqslant j \leqslant 1.\]
\end{lemma}

\begin{proof}
	Recall that the frequency vector $\nu$ satisfies the Diophantine condition
	\[|\langle k, \nu \rangle| > \frac{\kappa_\epsilon}{\|k\|_{\ell^1}^\tau}, \quad \forall k \in \mathbb{Z}^d \setminus \{\mathbf{0}\}, \]
	where $\kappa_\epsilon > 0$ and $\tau \geqslant d - 1$. We first verify that $\Omega_k = \langle k, \nu \rangle - \nu_1$ inherits a similar Diophantine property for all $k \in \Xi$. Letting $k^* := k - (1,0,\ldots,0) \in \mathbb{Z}^d \setminus \{\mathbf{0}\}$, we have
	\[\|k^*\|_{\ell^1} \leqslant \|k\|_{\ell^1} + \|(1,0,\ldots,0)\|_{\ell^1} = \|k\|_{\ell^1} + 1.\]
	It then follows that
	\begin{equation}\label{eq1}
		|\Omega_k| = |\langle k^*, \nu \rangle| \geqslant \frac{\kappa_\epsilon}{\|k^*\|_{\ell^1}^\tau} \geqslant \frac{\kappa_\epsilon}{(\|k\|_{\ell^1} + 1)^\tau}, \quad \forall k \in \Xi,
	\end{equation}
	as desired. For sufficiently large $T$, Lemma \ref{lemma:2} implies $|x(T)| \leqslant 1$. Consequently, noting that $1-\zeta\tau = \beta(\tau + \beta)^{-1}$, we deduce that for all $k \in \mathcal{S}_{\mathrm{T}} \setminus \{\mathbf{0}\}$,
	\begin{equation}\label{dio2}
		|x(T)+\Omega_k T| \geqslant |\Omega_k T| - |x(T)| \geqslant \frac{\kappa_\epsilon T}{(\|k\|_{\ell^1} + 1)^\tau} - 1 \geqslant \frac{\kappa_\epsilon}{2^{\tau + 1}} \frac{T}{T^{\zeta\tau}} = \frac{\kappa_\epsilon}{2^{\tau + 1}} T^{\beta(\tau + \beta)^{-1}}.
	\end{equation}
	For $k=\mathbf{0} \in \mathcal{S}_{\mathrm{T}}$, we similarly have, for sufficiently large $T$,
	\begin{equation}\label{dio3}
		|x(T) + \Omega_{\mathbf{0}} T| \geqslant |\Omega_{\mathbf{0}} T| - |x(T)| \geqslant |\Omega_{\mathbf{0}}|T - 1 \geqslant \frac{\kappa_\epsilon}{2^{\tau + 1}}T^{\beta(\tau + \beta)^{-1}}.
	\end{equation}
	These estimates allow us to circumvent the potential transition from a nonresonant to a resonant regime.
	
	From the definition of $\varphi$ and the fact that the weighting function $w \in C_0^\infty([-1,1])$, integrating by parts $\iota$ times yields
	\begin{align*}
		D^j\varphi (x(T) + \Omega_k T) &= \frac{i^j}{2}\int_{-1}^1 w(t)t^j \exp(i(x(T) + \Omega_k T)t)dt \\
		&= \frac{i^j(-1)^\iota}{2(i(x(T) + \Omega_k T))^\iota}\int_{-1}^1 D^\iota(w(t)t^j)\exp(i(x(T) + \Omega_k T)t)dt, \quad 0 \leqslant j \leqslant 1.
	\end{align*}	
	Hence, utilizing \eqref{dio2} and \eqref{dio3}, we find that for any $k \in \mathcal{S}_{\mathrm{T}}$ and any $\iota \in \mathbb{N}^+$,
	\begin{align}
		|D^j\varphi (x(T) + \Omega_k T)| &\leqslant \frac{1}{2|x(T) + \Omega_k T|^\iota}\int_{-1}^1 |D^\iota(w(t)t^j)||\exp(i(x(T) + \Omega_k T)t)|dt \notag \\
		&\leqslant \frac{1}{2}\left(\frac{2^{\tau + 1}}{\kappa_\epsilon}\right)^\iota\frac{1}{T^{\beta(\tau + \beta)^{-1}\iota}}\|D^\iota(w(t)t^j)\|_{L^1(-1,1)},\quad T \gg 1, \quad 0 \leqslant j \leqslant 1.\label{Djfuhe}
	\end{align}
	To further estimate \eqref{Djfuhe}, we require a key lemma that establishes the asymptotic behavior of the higher-order derivatives of the weighting function $w$. Recall \cite[Lemma 4.1]{TL25b}, which demonstrates the existence of a constant $\lambda = \lambda_{p,q} > 1$ such that for $\beta$ as defined above,\footnote{Although the weighting functions differ slightly in form, the analysis remains parallel.}
	\begin{equation}\label{gaodao}
		\|D^l w(t)\|_{L^1(-1,1)} \leqslant \lambda^l l^{\beta l}, \quad \forall l \in \mathbb{N}^+.
	\end{equation}
	By the general Leibniz rule, we observe that
	\[D^\iota(w(t)t^j) = \sum_{u = 0}^\iota \binom{\iota}{u}D^{\iota - u}w(t) D^u t^j = \sum_{u = 0}^j \binom{\iota}{u}D^{\iota - u}w(t) D^u t^j,\quad 0 \leqslant j \leqslant 1.\]
	Note that for $t \in [-1,1]$ and $0 \leqslant j \leqslant 1$, we have the bound $\|D^u t^j\|_{L^\infty(-1,1)} \leqslant 1$. Then, from \eqref{gaodao} and $\beta > 1$, we deduce that\footnote{The quantitative estimates obtained here clearly strengthen the qualitative results of \cite{TL25a} in the Ces\`aro sense.}
	\begin{align}
		\|D^\iota(w(t)t^j)\|_{L^1(-1,1)} &\leqslant \sum_{u = 0}^j \binom{\iota}{u}\|D^{\iota - u}w(t)\|_{L^1(-1,1)} \|D^u t^j\|_{L^\infty(-1,1)}\notag \\
		&\leqslant  \sum_{u = 0}^1 \binom{\iota}{u}\lambda^{\iota - u}(\iota - u)^{\beta (\iota - u)}\notag \\
		&\leqslant   \sum_{u = 0}^1 \iota^u \lambda^\iota \iota^{\beta (\iota - u)} \notag \\
		&\leqslant 2 \lambda^\iota \iota^{\beta\iota},\quad 0 \leqslant j \leqslant 1. \label{hhgd}
	\end{align}
	Substituting \eqref{hhgd} into \eqref{Djfuhe}, we obtain that for any $k \in \mathcal{S}_{\mathrm{T}}$, sufficiently large $T$ (independent of $k$), and any $\iota \in \mathbb{N}^+$ (possibly depending on $T$), there exists a constant $c_2 := 2^{\tau + 1}\kappa_\epsilon^{-1}\lambda$ such that
	\begin{equation}\label{zhongjgj}
		|D^j\varphi (x(T) + \Omega_k T)| \leqslant \left(\frac{2^{\tau + 1}}{\kappa_\epsilon}\right)^\iota\frac{\lambda^\iota \iota^{\beta \iota}}{T^{\beta(\tau + \beta)^{-1}\iota}} \leqslant \frac{c_2^\iota \iota^{\beta \iota}}{T^{\beta(\tau + \beta)^{-1}\iota}},\quad 0 \leqslant j \leqslant 1.
	\end{equation}
	
	To derive the exponential decay estimate from \eqref{zhongjgj}, we optimize the derivative order $\iota$ with respect to $T$ (denoted as $\iota^*$). For sufficiently large $T$, we choose the integer $\iota^*$ as
	\[ \iota^* := \lfloor c_3 T^\zeta \rfloor,\]
	where $\lfloor \cdot \rfloor$ denotes the floor function, and the constant $c_3 > 0$ is defined as $c_3 := e^{-1} c_2^{-1/\beta}$. Since $T$ is sufficiently large, we have $\iota^* \in \mathbb{N}^+$. By the definition of the floor function, we observe that
	\[c_3 T^\zeta - 1 < \iota^* \leqslant c_3 T^\zeta.\]
	Consequently, we can bound the term $(\iota^*)^{\beta \iota^*}$ as follows:
	\[(\iota^*)^{\beta \iota^*} \leqslant (c_3 T^\zeta)^{\beta \iota^*} = c_3^{\beta \iota^*} T^{\beta \zeta \iota^*}.\]
	Substituting this inequality into \eqref{zhongjgj} and noting that $\beta \zeta = \beta(\tau + \beta)^{-1}$, the algebraic powers of $T$ cancel out exactly:
	\[|D^j\varphi (x(T) + \Omega_k T)| \leqslant \frac{c_2^{\iota^*} (\iota^*)^{\beta \iota^*}}{T^{\beta (\tau + \beta)^{-1} \iota^*}} 
	\leqslant \frac{c_2^{\iota^*} c_3^{\beta \iota^*} T^{\beta \zeta \iota^*}}{T^{\beta \zeta \iota^*}} 
	= (c_2 c_3^\beta)^{\iota^*},\quad 0 \leqslant j \leqslant 1.\]
	With our choice of $c_3 = e^{-1} c_2^{-1/\beta}$, it follows that $c_2 c_3^\beta=e^{-\beta}$.
	Thus, the upper bound simplifies to
	\[|D^j\varphi (x(T) + \Omega_k T)| \leqslant \exp(-\beta \iota^*),\quad 0 \leqslant j \leqslant 1.\]
	Utilizing the lower bound $\iota^* > c_3 T^\zeta - 1$, we further obtain
	\[\exp(-\beta \iota^*) < \exp(-\beta (c_3 T^\zeta - 1)) = e^\beta \exp(-\beta c_3 T^\zeta).\]
	We then arrive at the following exponential decay estimate for $ T \gg 1 $:
	\begin{equation}\label{finaldecay}
		|D^j\varphi (x(T) + \Omega_k T)| \leqslant e^\beta \exp(-\beta c_3 T^\zeta),\quad \forall k \in \mathcal{S}_{\mathrm{T}}, \quad 0 \leqslant j \leqslant 1.
	\end{equation}
	
	Finally, utilizing the exponential decay \eqref{finaldecay} and the analyticity of $f$ (which guarantees the absolute convergence of its Fourier series, i.e., $\sum_{k \in \mathbb{Z}^d} |a_k| < \infty$), we obtain as $T \to +\infty$:
	\begin{align*}
		\left| \sum_{k \in \mathcal{S}_{\mathrm{T}}} a_k D^j\varphi (x(T) + \Omega_k T) \right| &\leqslant \left( \sum_{k \in \mathcal{S}_{\mathrm{T}}} |a_k| \right)\left( \max_{k \in \mathcal{S}_{\mathrm{T}}} | D^j\varphi (x(T) + \Omega_k T) | \right) \\
		&\leqslant \left( \sum_{k \in \mathbb{Z}^d} |a_k| \right) e^\beta \exp(-\beta c_3 T^\zeta) \\
		&\lesssim \exp(-\beta c_3 T^\zeta),\quad 0 \leqslant j \leqslant 1.
	\end{align*}
	Setting $c_1 := \beta c_3 > 0$ yields the desired estimate, which completes the proof of Lemma \ref{lemma:3}.
\end{proof}

\begin{lemma}\label{lemma:4}
	There exists a constant $c_4 > 0$ such that the following asymptotic estimate holds as $T \to +\infty$:
	\[\sum_{k \in \mathcal{S}_{\mathrm{R}}} a_k D^j\varphi(x(T) + \Omega_k T) = \mathcal{O}\big(\exp(-c_4 T^\zeta)\big), \quad 0 \leqslant j \leqslant 1.\]
\end{lemma}

\begin{proof}
	By analyticity, there exist constants $C_f, r > 0$ such that $|a_k| \leqslant C_f \exp(-r\|k\|_{\ell^1})$ for all $k \in \mathcal{S}_{\mathrm{R}}$. Consequently, utilizing the uniform bounds on $D^j\varphi$ ($0 \leqslant j \leqslant 1$) provided by Lemma \ref{lemma:1}, we obtain that for any $0 < c_4 < r$ and for $T$ sufficiently large,
	\begin{align*}
		\left| \sum_{k \in \mathcal{S}_{\mathrm{R}}} a_k D^j\varphi(x(T) + \Omega_k T) \right| 
		&\leqslant \sum_{k \in \mathcal{S}_{\mathrm{R}}} |a_k| \big| D^j\varphi(x(T) + \Omega_k T) \big| \\
		&\lesssim \sum_{k \in \mathcal{S}_{\mathrm{R}}} \exp(-r\|k\|_{\ell^1}) \\
		&\lesssim \int_{T^\zeta}^{+\infty} y^{d-1} e^{-ry} dy \\
		&\lesssim \exp(-c_4 T^\zeta), \quad 0 \leqslant j \leqslant 1.
	\end{align*}
	This completes the proof of Lemma \ref{lemma:4}.
\end{proof}

Substituting the exponential estimates from Lemmas \ref{lemma:3} and \ref{lemma:4} into \eqref{fenjie}, we obtain the following global exponential estimate as $T \to +\infty$:
\begin{equation}\label{gjgj}
	\sum_{k \in \Xi} a_k D^j\varphi(x(T) + \Omega_k T) = \mathcal{O}\big(\exp(-c_5 T^\zeta)\big), \quad 0 \leqslant j \leqslant 1,
\end{equation}
where the constant $c_5$ satisfies $c_5 = \min\{c_1, c_4\} > 0$.

Recall from Lemma \ref{lemma:1} that as $x \to 0$,
\[\varphi(x) = 1 + o(x) \quad \text{and} \quad D\varphi(x) = D^2\varphi(0)x + o(x).\]
Consequently, evaluating \eqref{def:F} at $x = x(T)$ and combining it with the uniform bounds on $D^j\varphi$ ($0 \leqslant j \leqslant 1$) from Lemma \ref{lemma:2} and the global exponential estimate in \eqref{gjgj}, we deduce that as $T \to +\infty$:
\[\begin{aligned}
	0 &= D\varphi(x(T))\varphi(x(T)) + \sum_{k \in \Xi} \Re(a_k) \big( \varphi(x(T))D\varphi(x(T) + \Omega_k T) + D\varphi(x(T))\varphi(x(T) + \Omega_k T) \big) \\
	&\quad + \sum_{k,l \in \Xi} \Re(a_k \bar{a}_l) \varphi(x(T) + \Omega_k T)D\varphi(x(T) + \Omega_l T) \\
	&= D^2\varphi(0)x(T) + o(x(T)) + \mathcal{O}\big(\exp(-c_5 T^\zeta)\big),
\end{aligned}\]
which yields
\begin{equation}\label{xT}
	x(T) = \mathcal{O}\big(\exp(-c_5 T^\zeta)\big),\quad T \to +\infty.
\end{equation}
Since $x(T) = (\nu_1 - \sigma)T$ with $\sigma = \nu_1^T$, we ultimately conclude that
\begin{equation}\label{bjgj}
	\nu_1^T - \nu_1 = \mathcal{O}\big(T^{-1}\exp(-c_5 T^\zeta)\big) = \mathcal{O}\big(\exp(-c_{\mathrm{I}} T^\zeta)\big),\quad T \to +\infty,
\end{equation}
where $c_{\mathrm{I}}:=c_5>0$.  This completes the proof of Theorem \ref{TH2}.

\subsection{Proof of Theorem \ref{THBRUNO}: Towards Brjuno nonresonance}\label{SUBSEC22}

Under the Brjuno nonresonance condition \eqref{BRU}--\eqref{BRU2}, namely
\[
|\langle k,\nu \rangle| > \frac{\kappa_{\epsilon}}{\Delta(\|k\|_{\ell^1})}, \quad \forall k \in \mathbb{Z}^d \setminus \{\mathbf{0}\},
\]
where $\Delta : [1,+\infty) \to [1,+\infty)$ is a monotonically increasing approximation function fulfilling
\[
\int_1^{+\infty} \frac{\log \Delta(x)}{x^2}dx < +\infty, \quad \frac{\log \Delta(x)}{x} \searrow 0 \text{\;\;as\;\;} x \to +\infty,
\]
the proof of the frequency map analysis differs from that in the Diophantine setting. Nevertheless, recalling the analysis in Section \ref{SEC21}, the essential differences arise only after Lemma \ref{lemma:2}. For convenience, we retain the notation introduced in Section \ref{SEC21}. Consequently, we proceed to investigate the asymptotic behavior of the series $\sum_{k \in \Xi} a_k D^j \varphi(x(T) + \Omega_k T)$ as $T \to +\infty$ for $0 \leqslant j \leqslant 1$, where $\Omega_k := \langle k, \nu \rangle - \nu_1$ for all $k \in \Xi$, and the continuous map $x(T)$ is as established in Lemma \ref{lemma:2}.

Let us first fix an arbitrary $\varepsilon>0$ and recall the constants $\lambda$ and $\beta$ defined in Section \ref{SEC21}, which will remain fixed throughout this paper. We then decompose the index set $\Xi$ into two disjoint parts:
\[
\widehat{\mathcal{S}}_{\mathrm{T}} := \left\{ k \in \Xi : \|k\|_{\ell^1} \leqslant K(T) \right\}, \quad \widehat{\mathcal{S}}_{\mathrm{R}} := \left\{ k \in \Xi : \|k\|_{\ell^1} > K(T) \right\},
\]
where the truncation function $K(T)$ is uniquely determined by the relation
\begin{equation}\label{rela}
	K^\beta(T) \exp\left(\frac{\varepsilon K(T)}{\log K(T)}\right) = T, \quad T \gg 1.
\end{equation}
It is evident that $K(T) \to +\infty$ as $T \to +\infty$, which necessarily implies
\begin{equation}\label{rela2}
	T \leqslant \exp\left(\frac{2\varepsilon K(T)}{\log K(T)}\right) \implies K(T) \geqslant \frac{c_6}{\varepsilon} \log T \log \log T, \quad T \gg 1
\end{equation}
for some constant $c_6>0$. Hereafter, all universal constants are independent of $\varepsilon$ unless explicitly stated otherwise. With this decomposition, we arrive at
\begin{align}
&\;\sum_{k \in \Xi} a_k D^j\varphi(x(T) + \Omega_k T) \notag \\
= &\;\sum_{k \in \widehat{\mathcal{S}}_{\mathrm{T}}} a_k D^j \varphi(x(T) + \Omega_k T) + \sum_{k \in \widehat{\mathcal{S}}_{\mathrm{R}}} a_k D^j \varphi(x(T) + \Omega_k T), \quad 0 \leqslant j \leqslant 1.\label{fenjie3}
\end{align}
As in Section \ref{SEC21}, we estimate the two resulting components using distinct methods.
\begin{lemma}\label{lemma:25}
	There exists a constant $c_7 > 0$ such that the following asymptotic estimate holds as $T \to +\infty$:
	\[
	\sum_{k \in \widehat{\mathcal{S}}_{\mathrm{T}}} a_k D^j \varphi(x(T) + \Omega_k T) = \mathcal{O}\big( \exp(-c_7 \varepsilon^{-1} \log T \log \log T) \big), \quad 0 \leqslant j \leqslant 1.
	\]
\end{lemma}
\begin{proof}
	Let us begin with a key observation: for any $\varepsilon > 0$, the approximation function $\Delta(x)$ in the Brjuno condition satisfies
	\begin{equation}\label{bgj}
		\Delta(x) \leqslant \exp\left( \frac{\varepsilon x}{\log x} \right),
	\end{equation}
	provided that $x$ is sufficiently large. To establish this, observe that the monotonicity of $\frac{\log \Delta(x)}{x}$ implies
	\[
	\int_{\log x}^x \frac{\log \Delta(y)}{y^2} dy \geqslant \frac{\log \Delta(x)}{x} \int_{\log x}^x \frac{1}{y} dy = \frac{\log \Delta(x)}{x} (\log x - \log \log x) \geqslant \frac{\log \Delta(x) \log x}{2x}, \quad x \gg 1.
	\]
	On the other hand, since $\int_1^{+\infty} \frac{\log \Delta(x)}{x^2} dx < +\infty$, the Cauchy criterion guarantees that
	\[
	\int_{\log x}^x \frac{\log \Delta(y)}{y^2} dy \leqslant \frac{\varepsilon}{2}, \quad x \gg 1.
	\]
	Combining these two estimates immediately yields the desired assertion.
	
	Next, we prove that there exists a constant $c_8 > 0$ such that for sufficiently large $T$,
	\[
	|x(T) + \Omega_k T| \geqslant c_8 T \exp\left( -\frac{\varepsilon K(T)}{\log K(T)} \right), \quad \forall k \in \widehat{\mathcal{S}}_{\mathrm{T}}.
	\]
	Note that there exists a constant $c_9 > 0$ such that $\Delta(x + 1) \leqslant c_9 \Delta(x)$ for all $x \geqslant 1$. Consequently, following the derivation of \eqref{eq1}, we obtain
	\[
	|\Omega_k| \geqslant \frac{\kappa_{\epsilon}}{\Delta(\|k\|_{\ell^1} + 1)} \geqslant \frac{\kappa_{\epsilon}}{c_9 \Delta(\|k\|_{\ell^1})}, \quad \forall k \in \Xi.
	\]
	Utilizing this inequality and noting from \eqref{rela} that $T \exp\left( -\frac{\varepsilon K(T)}{\log K(T)} \right) = K^\beta(T) \to +\infty$ as $T \to +\infty$, we deduce from \eqref{dio2} and \eqref{bgj} that for $T \gg 1$,
	\begin{align*}
		\left| x(T) + \Omega_k T \right| &\geqslant \left| \Omega_k T \right| - |x(T)|\\ &\geqslant \frac{\kappa_{\epsilon} T}{c_9 \Delta(\|k\|_{\ell^1})} - 1 \\
		&\geqslant \frac{\kappa_{\epsilon} T}{c_9 \Delta(K(T))} - 1 \\
		&\geqslant c_8 T \exp\left( -\frac{\varepsilon K(T)}{\log K(T)} \right), \quad \forall k \in \widehat{\mathcal{S}}_{\mathrm{T}} \setminus \{\mathbf{0}\},
	\end{align*}
	and from \eqref{dio3} that
	\[
	|x(T) + \Omega_{\mathbf{0}} T| \geqslant |\Omega_{\mathbf{0}} T| - |x(T)| \geqslant |\Omega_{\mathbf{0}}| T - 1 \geqslant c_8 T \exp\left( -\frac{\varepsilon K(T)}{\log K(T)} \right),
	\]
	provided that the constant $c_8 > 0$ satisfies $0 < c_8 < (2c_9)^{-1} \kappa_{\epsilon}$. This confirms the desired lower bound.
	
	With these bounds established, employing the method used to derive \eqref{Djfuhe} alongside \eqref{hhgd} yields that for any $k \in \widehat{\mathcal{S}}_{\mathrm{T}}$, for all sufficiently large $T$ (independent of $k$), and for any $\iota \in \mathbb{N}^+$ (possibly depending on $T$),
	\begin{align*}
		\left| D^j\varphi (x(T) + \Omega_k T) \right| &\leqslant \frac{1}{2|x(T) + \Omega_k T|^\iota} \int_{-1}^1 \left| D^\iota \big( w(t)t^j \big) \right| \left| \exp(i(x(T) + \Omega_k T)t) \right| dt \\
		&\leqslant c_{10}^\iota T^{-\iota} \exp\left( \frac{\varepsilon K(T)\iota}{\log K(T)} \right) \iota^{\beta \iota}, \quad 0 \leqslant j \leqslant 1,
	\end{align*}
	where $c_{10} := \lambda c_8^{-1} > 0$. Furthermore, we choose the integer
	\[
	\iota^* = \left\lfloor c_{10}^{-\frac{1}{\beta}} e^{-1} T^{\frac{1}{\beta}} \exp\left( -\frac{\varepsilon K(T)}{\beta \log K(T)} \right) \right\rfloor
	\]
	to minimize the upper bound. Following the optimization procedure detailed in Lemma \ref{lemma:3} (specifically for \eqref{zhongjgj}), and appealing to \eqref{rela} and \eqref{rela2}, we ultimately arrive at the following estimate for $T \gg 1$:
	\begin{align*}
		\left| D^j\varphi (x(T) + \Omega_k T) \right| &\leqslant \exp\left( -c_{11} T^{\frac{1}{\beta}} \exp\left( -\frac{\varepsilon K(T)}{\beta \log K(T)} \right) \right)\\
		 &= \exp(-c_{11}K(T)) \\
		 &\leqslant \exp(-c_7 \varepsilon^{-1} \log T \log \log T),\quad  \forall k \in \widehat{\mathcal{S}}_{\mathrm{T}}, \quad 0 \leqslant j \leqslant 1,
	\end{align*}
 where the constant $c_{11}$ satisfies $0 < c_{11} < \beta e^{-1} c_{10}^{-1/\beta}$ and we set $c_7 := c_6 c_{11} > 0$. Finally, summing over $k \in \widehat{\mathcal{S}}_{\mathrm{T}}$ and taking into account the absolute summability bound $\sum_{k \in \widehat{\mathcal{S}}_{\mathrm{T}}} |a_k| \leqslant \sum_{k \in \mathbb{Z}^d} |a_k| < +\infty$, the proof of Lemma \ref{lemma:25} is complete.
\end{proof}

\begin{lemma}\label{lemma:26}
	There exists a constant $c_{12} > 0$ such that the following asymptotic estimate holds as $T \to +\infty$:
	\[
	\sum_{k \in \widehat{\mathcal{S}}_{\mathrm{R}}} a_k D^j \varphi(x(T) + \Omega_k T) = \mathcal{O}\big( \exp(-c_{12} \varepsilon^{-1} \log T \log \log T) \big), \quad 0 \leqslant j \leqslant 1.
	\]
\end{lemma}
\begin{proof}
	Analogous to the analysis in Lemma \ref{lemma:4}, we immediately deduce from \eqref{rela2} that as $T \to +\infty$,
	\begin{align*}
	\left| \sum_{k \in \widehat{\mathcal{S}}_{\mathrm{R}}} a_k D^j \varphi(x(T) + \Omega_k T) \right| &\lesssim \int_{K(T)}^{+\infty} y^{d-1} e^{-ry} dy \\
	&\lesssim \exp(-c_{13} K(T)) \\
	&\leqslant \exp(-c_{12} \varepsilon^{-1} \log T \log \log T), \quad 0 \leqslant j \leqslant 1, 	
	\end{align*}
	where the constant $c_{13}$ satisfies $0 < c_{13} < r$, and we set $c_{12} := c_6 c_{13} > 0$. This completes the proof of Lemma \ref{lemma:26}.
\end{proof}

Combining Lemmas \ref{lemma:25} and \ref{lemma:26}, we obtain from \eqref{fenjie3} the following exponential estimate as $T \to +\infty$:
\[
\sum_{k \in \Xi} a_k D^j \varphi(x(T) + \Omega_k T) = \mathcal{O}\big( \exp(-c_{14} \varepsilon^{-1} \log T \log \log T) \big), \quad 0 \leqslant j \leqslant 1,
\]
where the constant $c_{14}$ is defined as $c_{14} := \min\{c_7, c_{12}\} > 0$. Consequently, following an analysis analogous to that in \eqref{xT} and \eqref{bjgj}, we finally arrive at
\begin{align*}
\nu_1^T - \nu_1 & = \mathcal{O}\big(T^{-1}x(T)\big)\\
& = \mathcal{O}\big( T^{-1} \exp(-c_{14} \varepsilon^{-1} \log T \log \log T) \big) \\
&= \mathcal{O}\big( \exp(-c_{\mathrm{II}} \log T \log \log T) \big), \quad T \to +\infty,
\end{align*}
where $c_{\mathrm{II}} := c_{14} \varepsilon^{-1} > 0$. Since $\varepsilon > 0$ is arbitrary, this completes the proof of Theorem \ref{THBRUNO}.

\subsection{Proof of Theorem \ref{TH3}: Exponential convergence for frequency map analysis via almost periodicity}\label{SUBSEC23}
The proof of the frequency map analysis in the almost periodic case differs from that in the quasi-periodic setting in several respects. Nevertheless, the overarching methodology remains analogous, as demonstrated in Section \ref{SUBSEC22}. We shall now estimate the asymptotic behavior of the series $\sum_{k \in \Theta} a_k D^j \varphi(x(T) + \Omega_k T)$ as $T \to +\infty$ for $0 \leqslant j \leqslant 1$, where $\Theta := \mathbb{Z}_*^\mathbb{N} \setminus \{(1,0,\dots)\}$ and $\Omega_k := \langle k, \nu \rangle - \nu_1$ for all $k \in \Theta$, utilizing the continuous map $x(T)$ established in Lemma \ref{lemma:2}. In this context, we recall that the infinite-dimensional Diophantine condition in \eqref{DIO2} for the frequency vector $\nu$ reads
\[
|\langle k, \nu \rangle| > \frac{\gamma}{\prod\limits_{j \in \mathbb{N}} (1 + \langle j \rangle^\mu |k_j|^\mu)}, \quad \forall k \in \mathbb{Z}^\mathbb{N} \text{\;\;with\;\;} 0 < \sum\limits_{j \in \mathbb{N}} |k_j| < +\infty,
\]
for some constants $\gamma > 0$ and $\mu > 1$.

Let us decompose the index set $\Theta$ into two disjoint subsets:
\[
\widecheck{\mathcal{S}}_{\mathrm{T}} := \left\{ k \in \Theta :\; |k|_{\eta} \leqslant (\log T)^{\rho} \right\}, \quad \widecheck{\mathcal{S}}_{\mathrm{R}} := \left\{ k \in \Theta :\; |k|_{\eta} > (\log T)^{\rho} \right\},
\]
where $|k|_\eta := \sum_{j\in\mathbb{N}} \langle j \rangle^\eta |k_j|$ for a fixed integer $\eta \geqslant 2$, and the constant $\rho$ is chosen to satisfy $2 \leqslant \rho < 1+\eta$. Consequently, we have
\begin{align}
	&\;	\sum_{k \in \Theta} a_k D^j\varphi(x(T) + \Omega_k T) \notag \\
	= &\;\sum_{k \in \widecheck{\mathcal{S}}_{\mathrm{T}}} a_k D^j \varphi(x(T) + \Omega_k T) + \sum_{k \in \widecheck{\mathcal{S}}_{\mathrm{R}}} a_k D^j \varphi(x(T) + \Omega_k T), \quad 0 \leqslant j \leqslant 1.\label{fenjie2}
\end{align}
As in Section \ref{SEC21}, we continue to estimate the resulting components using distinct approaches, wherein the infinite-dimensional spatial structure will introduce fundamental differences from the finite-dimensional setting.

\begin{lemma}\label{lemma:5}
	There exists a constant $c_{15} > 0$ such that the following asymptotic estimate holds as $T \to +\infty$:
	\[ 
	\sum_{k \in \widecheck{\mathcal{S}}_{\mathrm{T}}} a_k D^j\varphi(x(T) + \Omega_k T) = \mathcal{O}\Big(\exp\big(-c_{15} (\log T)^\rho\big)\Big), \quad 0 \leqslant j \leqslant 1. 
	\]
\end{lemma}

\begin{proof}
	To establish the desired exponential decay estimate, we first precisely estimate the growth order arising from the infinite-dimensional small divisors. For any fixed $k \in \widecheck{\mathcal{S}}_{\mathrm{T}}$, let $\mathscr{N}$ denote the number of its non-zero components, whose indices are ordered as $j_1 < j_2 < \cdots < j_{\mathscr{N}}$. Clearly, $\langle j_s \rangle \geqslant s$ for all $1 \leqslant s \leqslant \mathscr{N}$. Since the non-zero integer components satisfy $|k_{j_s}| \geqslant 1$, it follows from the definition of the metric $|\cdot|_\eta$ and monotonicity that
	\begin{align*}
		\mathscr{N} &= \left( (1 + \eta)\int_0^\mathscr{N} y^\eta \,dy \right)^{\frac{1}{1 + \eta}} \leqslant \left( (1 + \eta)\sum_{s = 1}^\mathscr{N} s^\eta \right)^{\frac{1}{1 + \eta}} \leqslant \left( (1 + \eta)\sum_{s = 1}^\mathscr{N} \langle j_s \rangle^\eta \right)^{\frac{1}{1 + \eta}} \\
		&\leqslant \left( (1 + \eta)\sum_{s = 1}^\mathscr{N} \langle j_s \rangle^\eta |k_{j_s}| \right)^{\frac{1}{1 + \eta}} = \left( (1 + \eta) |k|_\eta \right)^{\frac{1}{1 + \eta}} \leqslant (1 + \eta)^{\frac{1}{1 + \eta}} (\log T)^{\frac{\rho}{1 + \eta}}.
	\end{align*}
	Furthermore, since $\eta \geqslant 2$, we have
	\[ 
	\langle j_s \rangle^\mu |k_{j_s}|^\mu = (\langle j_s \rangle |k_{j_s}|)^\mu \leqslant (\langle j_s \rangle^\eta |k_{j_s}|)^\mu \leqslant \Big( \sum_{j \in \mathbb{N}} \langle j \rangle^\eta |k_j| \Big)^\mu = |k|_\eta^\mu \leqslant (\log T)^{\rho\mu}, \quad 1 \leqslant s \leqslant \mathscr{N}. 
	\]
	Combining the above estimates yields that for $T$ sufficiently large,
	\begin{align}
		\max_{k \in \widecheck{\mathcal{S}}_{\mathrm{T}}} \prod_{j \in \mathbb{N}} \big( 1 + \langle j \rangle^\mu |k_j|^\mu \big) &= \max_{k \in \widecheck{\mathcal{S}}_{\mathrm{T}}} \exp\Big( \sum_{s = 1}^\mathscr{N} \log (1 + \langle j_s \rangle^\mu |k_{j_s}|^\mu ) \Big) \notag\\
		&\leqslant \max_{k \in \widecheck{\mathcal{S}}_{\mathrm{T}}} \exp\Big( \sum_{s = 1}^\mathscr{N} \log ( 1 + (\log T)^{\rho\mu} ) \Big) \notag\\
		&\leqslant \exp\Big( c_{16} \max_{k \in \widecheck{\mathcal{S}}_{\mathrm{T}}} \mathscr{N} \log \log T \Big) \notag \\
		&\leqslant \exp\Big( c_{17} (\log T)^{\frac{\rho}{1 + \eta}} \log \log T \Big), \label{wqdio}
	\end{align}
	provided that we choose constants $c_{16} > 2\rho \mu$ and $c_{17} > (1 + \eta)^{1/(1 + \eta)} c_{16}$.
	
	By an analysis analogous to that in Lemmas \ref{lemma:3} and \ref{lemma:25}, one can verify that the term $\left| x(T) + \Omega_k T \right|$ exhibits infinite-dimensional Diophantine nonresonance. Specifically, in view of \eqref{wqdio}, there exists a constant $c_{18} > 0$ such that for $T \gg 1$, we have
	\[ 
	\left| x(T) + \Omega_k T \right| \geqslant c_{18} T \exp\Big(- c_{17} (\log T)^{\frac{\rho}{1 + \eta}} \log \log T\Big), \quad \forall k \in \widecheck{\mathcal{S}}_{\mathrm{T}}. 
	\]
	Consequently, for any $k \in \widecheck{\mathcal{S}}_{\mathrm{T}}$ and all sufficiently large $T$ (independent of $k$), the following estimates hold for any positive integer $\iota$ (which may depend on $T$):
	\[ 
	\left| D^j \varphi(x(T) + \Omega_k T) \right| \leqslant c_{19}^\iota T^{- \iota} \exp\Big(c_{17} (\log T)^{\frac{\rho}{1 + \eta}} (\log \log T) \iota\Big) \iota^{\beta \iota}, \quad 0 \leqslant j \leqslant 1, 
	\]
	where $c_{19} := \lambda c_{18}^{-1} > 0$. Optimizing this bound by choosing the integer
	\[ 
	\iota^* = \left\lfloor c_{19}^{- \frac{1}{\beta}} e^{-1} T^{\frac{1}{\beta}} \exp\Big(- \beta^{-1}c_{17} (\log T)^{\frac{\rho}{1 + \eta}} \log \log T\Big) \right\rfloor, 
	\]
	yields the following uniform exponential estimate as $T \to +\infty$:
	\begin{align}
		\left| D^j \varphi(x(T) + \Omega_k T) \right| &\leqslant \exp\bigg(- c_{20} T^{\frac{1}{\beta}} \exp\Big(- \beta^{-1}c_{17} (\log T)^{\frac{\rho}{1 + \eta}} \log \log T\Big)\bigg) \notag \\
		&\leqslant \exp\big(- c_{21} (\log T)^\rho\big), \label{zsgj3}	
	\end{align}
	for all $k \in \widecheck{\mathcal{S}}_{\mathrm{T}}$ and $0 \leqslant j \leqslant 1$, where $0 < c_{20} < \beta e^{-1} c_{19}^{-1/\beta}$ and $c_{21} > 0$ is a suitable constant.
	
	To obtain the exponential estimate for the sum over all $k \in \widecheck{\mathcal{S}}_{\mathrm{T}}$, it remains to establish a fundamental summability result on the infinite-dimensional lattice. Distinct from the finite-dimensional setting, the infinite-dimensional spatial structure necessitates precise cardinality estimates. By the analyticity of $f$, there exist constants $C_f, r > 0$ such that $|a_k| \leqslant C_f \exp(-r|k|_\eta)$ for all $k \in \mathbb{Z}_*^\mathbb{N}$. Recall from \cite[Lemma 8.4]{TL24b} that for sufficiently large integers $\vartheta$,
	\[ 
	\sum_{k \in \Theta, |k|_\eta = \vartheta} 1 := \# \{k \in \Theta :\; |k|_\eta = \vartheta\} \lesssim \vartheta^{\vartheta^{\frac{1}{\eta}}}. 
	\]
	Applying this bound and utilizing the condition $\eta \geqslant 2$, we deduce that for any constant $c_{22} > 0$,
	\begin{align}
		\sum_{k \in \Theta } {|a_k|} &\lesssim \sum_{k \in \Theta} \exp(-c_{22}|k|_\eta) 
		\lesssim \sum_{\vartheta=1}^\infty \sum_{k \in \Theta, |k|_\eta = \vartheta} \exp(-c_{22}|k|_\eta) \notag\\
		&= \sum_{\vartheta=1}^\infty \left( \sum_{k \in \Theta, |k|_\eta = \vartheta} 1 \right) \exp(-c_{22}\vartheta) 
		\lesssim \sum_{\vartheta=1}^\infty \vartheta^{\vartheta^{\frac{1}{\eta}}} \exp(-c_{22}\vartheta) \notag\\
		&= \sum_{\vartheta=1}^\infty \exp\big(-c_{22}\vartheta + \vartheta^{\frac{1}{\eta}}\log \vartheta\big) 
		\leqslant \sum_{\vartheta=1}^\infty \exp\big(-c_{22}\vartheta + \vartheta^{\frac{1}{2}}\log \vartheta\big) \notag \\
		&\lesssim \sum_{\vartheta=1}^\infty \exp\big(-\frac{c_{22}}{2}\vartheta\big) < +\infty. \label{333}
	\end{align}
	
	Since the sum $\sum_{k \in \widecheck{\mathcal{S}}_{\mathrm{T}}} |a_k|$ is dominated by \eqref{333}, we ultimately conclude that as $T \to +\infty$,
	\[ 
	\sum_{k \in \widecheck{\mathcal{S}}_{\mathrm{T}}} a_k D^j\varphi(x(T) + \Omega_k T) = \mathcal{O}\Big(\exp\big(-c_{15} (\log T)^\rho\big)\Big), \quad 0 \leqslant j \leqslant 1. 
	\]
	Setting $c_{15} := c_{21} > 0$ then yields the desired result, which completes the proof of Lemma \ref{lemma:5}.
\end{proof}

\begin{lemma}\label{lemma:6}
	There exists a constant $c_{22} > 0$ such that the following asymptotic estimate holds as $T \to +\infty$:
	\[\sum_{k \in \widecheck{\mathcal{S}}_{\mathrm{R}}} a_k D^j\varphi(x(T) + \Omega_k T) = \mathcal{O}\Big(\exp\big(-c_{22} (\log T)^\rho\big)\Big), \quad 0 \leqslant j \leqslant 1.\]
\end{lemma}
\begin{proof}
	By Lemma \ref{lemma:1}, and noting that $|k|_\eta > (\log T)^\rho$ holds for all $k \in \widecheck{\mathcal{S}}_{\mathrm{R}}$, we deduce that for any $0 < c_{22} < r$, the following estimate holds as $T \to +\infty$:
	\begin{align}
		\left| \sum_{k \in \widecheck{\mathcal{S}}_{\mathrm{R}}} a_k D^j \varphi(x(T) + \Omega_k T) \right| 
		&\leqslant \sum_{k \in \widecheck{\mathcal{S}}_{\mathrm{R}}} |a_k| \big| D^j\varphi(x(T) + \Omega_k T) \big|\notag \\
		&\lesssim \sum_{k \in \widecheck{\mathcal{S}}_{\mathrm{R}}} \exp(-r|k|_\eta) \notag \\
		&= \sum_{k \in \widecheck{\mathcal{S}}_{\mathrm{R}}} \exp(-c_{22}|k|_\eta) \exp\big(-(r - c_{22})|k|_\eta\big) \notag \\
		&\leqslant \exp\big(-c_{22}(\log T)^\rho\big) \sum_{k \in \Theta} \exp\big(-(r - c_{22})|k|_\eta\big) \notag \\
		&\lesssim \exp\big(-c_{22}(\log T)^\rho\big), \quad 0 \leqslant j \leqslant 1, \notag
	\end{align}
	where the last inequality follows from \eqref{333}. This completes the proof of Lemma \ref{lemma:6}.
\end{proof}

In view of \eqref{fenjie2}, an application of Lemmas \ref{lemma:5} and \ref{lemma:6} yields the following exponential decay bound as $T \to +\infty$:
\[
\sum_{k \in \Theta} a_k D^j \varphi(x(T) + \Omega_k T) = \mathcal{O}\Big(\exp\big(-c_{23} (\log T)^\rho\big)\Big), \quad 0 \leqslant j \leqslant 1,
\]
where we have set $c_{23} := \min\{c_{15}, c_{22}\} > 0$. Following reasoning similar to that used for \eqref{xT} and \eqref{bjgj}, we ultimately deduce that
\[
\nu_1^T - \nu_1 = \mathcal{O}\big(T^{-1}x(T)\big) = \mathcal{O}\Big(T^{-1} \exp\big(-c_{23} (\log T)^\rho\big)\Big) = \mathcal{O}\Big(\exp\big(-c_{\mathrm{III}} (\log T)^\rho\big)\Big)
\]
as $T \to +\infty$. Here, $2 \leqslant \rho < 1+\eta$ can be chosen arbitrarily, and $ c_{\mathrm{III}} := c_{23}>0$. This concludes the proof of Theorem \ref{TH3}.

\subsection{Proof of Theorem \ref{TH4}: Towards a general lattice setting}\label{SUBSEC24}
We first emphasize that, although both concern almost periodic frequency map analysis, the analysis presented in this section differs from that of Section \ref{SUBSEC23} in several respects. As before, we aim to estimate the asymptotic behavior of the series $\sum_{k \in \Upsilon} a_k D^j \varphi(x(T) + \Omega_k T)$ as $T \to +\infty$ for $0 \leqslant j \leqslant 1$, where $\Upsilon := \mathbb{Z}_{\Phi}^{\mathbb{N}} \setminus \{(1,0,\dots)\}$ and $\Omega_k := \langle k, \nu \rangle - \nu_1$ for all $k \in \Upsilon$. Here, the function $\Phi$ represents the weight in the weighted norm $\|k\|_\Phi := \sum_{j \in \mathbb{N}} \Phi (\langle j \rangle) |k_j|$, and the continuous map $x(T)$ is as obtained in Lemma \ref{lemma:2}.

We continue to utilize the infinite-dimensional Diophantine condition given by \eqref{DIO2} for the frequency vector $\nu$:
\[
|\langle k, \nu \rangle| > \frac{\gamma}{\prod\limits_{j \in \mathbb{N}} (1 + \langle j \rangle^\mu |k_j|^\mu)}, \quad \forall k \in \mathbb{Z}^\mathbb{N} \text{\;\;with\;\;} 0 < \sum\limits_{j \in \mathbb{N}} {|{k_j}|} < + \infty ,
\]
where $\gamma > 0$ and $\mu > 1$. Observe that, by definition, this condition is independent of the specific spatial structure; however, the subsequent small divisor estimates heavily rely on it, as will be shown in the sequel.

We decompose the index set $\Upsilon$ into two disjoint subsets:
\[
\widetilde{\mathcal{S}}_{\mathrm{T}} := \big\{ k \in \Upsilon :\; \|k\|_\Phi \leqslant \mathscr{K}(T) \big\},\quad \widetilde{\mathcal{S}}_{\mathrm{R}} := \big\{ k \in \Upsilon :\; \|k\|_\Phi > \mathscr{K}(T) \big\},
\]
where the truncation function $\mathscr{K}(T)$ will be determined later. According to this decomposition, we have
\begin{align}
	&\;\sum_{k \in \Upsilon} a_k D^j\varphi(x(T) + \Omega_k T)\notag  \\
	= &\;\sum_{k \in \widetilde{\mathcal{S}}_{\mathrm{T}}} a_k D^j \varphi(x(T) + \Omega_k T) + \sum_{k \in \widetilde{\mathcal{S}}_{\mathrm{R}}} a_k D^j \varphi(x(T) + \Omega_k T), \quad 0 \leqslant j \leqslant 1.\label{fjzh} 
\end{align}

To estimate the aforementioned terms via distinct approaches, analogous to those in Section \ref{SEC21}, we require two key lemmas, wherein the general spatial structure plays a fundamental role.

\begin{lemma}\label{lemma:9}
	Suppose that there exists a constant $b \gg 1$ such that for all sufficiently large $T$,
	\[T \gg \exp\big(b\Gamma(\mathscr{K}(T))\big),\]
	where $\Gamma(x) := \Psi(x)\log\Lambda(x)$, $\Psi(x)$ is the inverse of the map $x \mapsto \int_0^x \Phi(y)\,dy$, and
	\[
	\Lambda(x) := \begin{cases}
		x, & \text{if\;\;} \Phi(x) \geqslant x, \\
		\Phi^{-1}(x), & \text{if\;\;} \Phi(x) < x.
	\end{cases}
	\]
	Then there exist constants $c_{24},c_{25} > 0$ such that the following asymptotic estimate holds as $T \to +\infty$:
	\[\sum_{k \in \widetilde{\mathcal{S}}_{\mathrm{T}}} a_k D^j\varphi(x(T) + \Omega_k T) = \mathcal{O}\Big(\exp\Big(-c_{24}T^{\frac{1}{\beta}}\exp\big(-c_{25}\Gamma(\mathscr{K}(T))\big)\Big)\Big), \quad 0 \leqslant j \leqslant 1.\]
\end{lemma}
\begin{proof}
	We begin by establishing the small divisor estimate arising from the infinite-dimensional Diophantine condition within the general spatial framework. While the analysis closely parallels that of Lemma \ref{lemma:5}, it involves several fundamental distinctions. For any fixed $k \in \widetilde{\mathcal{S}}_{\mathrm{T}}$, let $\mathscr{M}$ denote the number of its non-zero components, indexed such that $j_1 < j_2 < \cdots < j_{\mathscr{M}}$. Note that $\langle j_s \rangle \geqslant s$ for all $1 \leqslant s \leqslant \mathscr{M}$. Thus, we obtain
	\[\mathscr{K}(T) \geqslant \| k \|_\Phi = \sum_{s = 1}^\mathscr{M} \Phi (\langle j_s \rangle)|k_{j_s}| \geqslant \sum_{s = 1}^{\mathscr{M}} \Phi (\langle j_s \rangle) \geqslant \int_0^{\mathscr{M}} \Phi (y)\,dy.\]
	By the definition of $\Psi(x)$, it immediately follows that
	\begin{equation}\label{Mgj}
		\mathscr{M} \leqslant \Psi (\mathscr{K}(T)).
	\end{equation}
	Next, we estimate $\langle j_s \rangle |k_{j_s}|$ for all $1 \leqslant s \leqslant \mathscr{M}$. To this end, we distinguish two cases based on the growth of $\Phi(x)$. If $\Phi(x) \geqslant x$, we have
	\begin{equation}\label{jkgj1}
		\langle j_s \rangle |k_{j_s}| \leqslant \Phi(\langle j_s \rangle)|k_{j_s}| \leqslant \sum_{s = 1}^{\mathscr{M}} \Phi(\langle j_s \rangle)|k_{j_s}| = \|k\|_\Phi \leqslant \mathscr{K}(T),\quad 1 \leqslant s \leqslant \mathscr{M}.
	\end{equation}
	Conversely, if $\Phi(x) < x$, it follows that $\Phi(\langle j_s \rangle) \leqslant \|k\|_\Phi \leqslant \mathscr{K}(T)$, which implies $\langle j_s \rangle \leqslant \Phi^{-1}(\mathscr{K}(T))$. Noting that $|k_{j_s}| \leqslant \|k\|_\Phi \leqslant \mathscr{K}(T)$, we deduce that
	\begin{equation}\label{jkgj2}
		\langle j_s \rangle |k_{j_s}| \leqslant \Phi^{-1}(\mathscr{K}(T))\mathscr{K}(T),\quad 1 \leqslant s \leqslant \mathscr{M}.
	\end{equation}
	Utilizing \eqref{jkgj1} and \eqref{jkgj2}, we obtain from the definition of $\Lambda(x)$ that, for sufficiently large $T$,
	\begin{align}
		&\;\max \big\{ \log \big(1 + (\mathscr{K}(T))^\mu \big), \log \big(1 + (\Phi^{-1}(\mathscr{K}(T))\mathscr{K}(T))^\mu \big)\big\} \notag \\
		 \leqslant &\;2\mu \max \big\{ \log \mathscr{K}(T), \log (\Phi^{-1}(\mathscr{K}(T))\mathscr{K}(T))\big\} \notag \\
		 \leqslant &\;4\mu \max \big\{ \log \mathscr{K}(T), \log \Phi^{-1}(\mathscr{K}(T))\big\} \notag \\
	  \leqslant &\;4\mu \log \Lambda (\mathscr{K}(T)). \label{dsgj}
	\end{align}
	Accordingly, for sufficiently large $T$, we deduce from \eqref{Mgj}, \eqref{dsgj}, and the definition of $\Gamma(x)$ that
	\begin{align}
		\max_{k \in \widetilde{\mathcal{S}}_{\mathrm{T}}} \prod_{j \in \mathbb{N}} (1 + \langle j \rangle^\mu |k_j|^\mu) &= \max_{k \in \widetilde{\mathcal{S}}_{\mathrm{T}}} \exp\bigg(\sum_{s = 1}^{\mathscr{M}} \log (1 + \langle j_s \rangle^\mu |k_{j_s}|^\mu)\bigg) \notag \\
		&\leqslant \max_{k \in \widetilde{\mathcal{S}}_{\mathrm{T}}} \exp\bigg(\sum_{s = 1}^{\mathscr{M}} \max \big\{ \log \big(1 + (\mathscr{K}(T))^\mu \big), \log \big(1 + (\Phi^{-1}(\mathscr{K}(T))\mathscr{K}(T))^\mu \big)\big\} \bigg) \notag \\
		&\leqslant \exp\Big(4\mu \max_{k \in \widetilde{\mathcal{S}}_{\mathrm{T}}} \mathscr{M}\log \Lambda (\mathscr{K}(T))\Big) \notag \\
		&\leqslant \exp\big(4\mu \Psi (\mathscr{K}(T))\log \Lambda (\mathscr{K}(T))\big) \notag \\
		&= \exp\big(c_{26}\Gamma (\mathscr{K}(T))\big), \label{wqgjyb}
	\end{align}
	where $c_{26} := 4\mu > 0$.
	
	Applying \eqref{wqgjyb} and proceeding as in Lemma \ref{lemma:5}, we deduce the existence of a constant $c_{27} > 0$ such that for $T \gg 1$, the following nonresonance condition holds:
	\[|x(T) + \Omega_k T| \geqslant c_{27}T\exp\big(-c_{26}\Gamma(\mathscr{K}(T))\big),\quad \forall k \in \widetilde{\mathcal{S}}_{\mathrm{T}}.\]
	Therefore, for any $k \in \widetilde{\mathcal{S}}_{\mathrm{T}}$, all sufficiently large $T$ (independent of $k$), and any positive integer $\iota$ (which may depend on $T$), the following bound holds:
	\[|D^j\varphi (x(T) + \Omega_k T)| \leqslant c_{28}^\iota T^{-\iota }\exp\big(c_{26}\Gamma(\mathscr{K}(T))\iota\big)\iota^{\beta \iota},\quad 0 \leqslant j \leqslant 1,\]
	where $c_{28} := \lambda c_{27}^{-1} > 0$. Under the assumptions of Lemma \ref{lemma:9}, optimizing $\iota$ by choosing
	\[\iota^* = \left\lfloor c_{28}^{-\frac{1}{\beta}}e^{-1}T^{\frac{1}{\beta}}\exp\Big(-\frac{c_{26}}{\beta}\Gamma(\mathscr{K}(T))\Big) \right\rfloor\]
	yields the following uniform estimate as $T \to +\infty$:
	\begin{equation}\label{dxgj}
		|D^j\varphi (x(T) + \Omega_k T)| \leqslant \exp\Big(-c_{24}T^{\frac{1}{\beta}}\exp\big(-c_{25}\Gamma(\mathscr{K}(T))\big)\Big),\quad \forall k \in \widetilde{\mathcal{S}}_{\mathrm{T}},\quad 0 \leqslant j \leqslant 1,
	\end{equation}
	where $0 < c_{24} < \beta e^{-1}c_{28}^{-1/\beta}$ and $c_{25} := \beta^{-1}c_{26}>0$.
	
	Finally, summing \eqref{dxgj} over $k \in \widetilde{\mathcal{S}}_{\mathrm{T}}$ and invoking the analyticity of $f$ (which implies that $\sum_{k \in \mathbb{Z}_\Phi^\mathbb{N}} |a_k| < +\infty$), we complete the proof of Lemma \ref{lemma:9}.
\end{proof}

\begin{lemma}\label{lemma:10}
	There exists a constant $c_{29} > 0$ such that the following asymptotic estimate holds as $T \to +\infty$:
	\[\sum_{k \in \widetilde{\mathcal{S}}_{\mathrm{R}}} a_k D^j\varphi(x(T) + \Omega_k T) = \mathcal{O}\Big(\exp\big(-c_{29}\mathscr{K}(T)\big)\Big), \quad 0 \leqslant j \leqslant 1.\]
\end{lemma}
\begin{proof}
	Recall that the analyticity of $f$ implies that there exists $r > 0$ such that 
	\[\sum_{k \in \mathbb{Z}_\Phi^\mathbb{N}} |a_k| \exp(r \|k\|_\Phi) < +\infty.\]
	This analyticity allows us to dispense with the cardinality estimates required in Lemma \ref{lemma:6} (see, e.g., \eqref{333}) and directly establish rapid decay, marking a crucial distinction. Therefore, by Lemma \ref{lemma:1} and the definition of the set $\widetilde{\mathcal{S}}_{\mathrm{R}}$, setting $c_{29} := r>0$, we deduce that, as $T \to +\infty$,
	\begin{align*}
		\left| \sum_{k \in \widetilde{\mathcal{S}}_{\mathrm{R}}} a_k D^j \varphi(x(T) + \Omega_k T) \right| 
		&\leqslant \sum_{k \in \widetilde{\mathcal{S}}_{\mathrm{R}}} |a_k| \big| D^j\varphi(x(T) + \Omega_k T) \big| \\
		&\leqslant \sum_{k \in \widetilde{\mathcal{S}}_{\mathrm{R}}} |a_k| \exp(r \|k\|_\Phi) \exp(-r \|k\|_\Phi) \\
		&\leqslant \bigg( \sum_{k \in \mathbb{Z}_\Phi^\mathbb{N}} |a_k| \exp(r \|k\|_\Phi) \bigg) \exp\big(-r\mathscr{K}(T)\big) \\
		&\lesssim \exp\big(-c_{29}\mathscr{K}(T)\big), \quad 0 \leqslant j \leqslant 1.
	\end{align*}
	This confirms the desired asymptotic behavior and completes the proof of Lemma \ref{lemma:10}.
\end{proof}

We now determine the truncation function $\mathscr{K}(T)$ by appealing to the estimates provided in Lemmas \ref{lemma:9} and \ref{lemma:10}. For sufficiently large $T$, set
\begin{equation}\label{ddgx}
	T^{\frac{1}{\beta}}\exp\big(-c_{25}\Gamma(\mathscr{K}(T))\big) = \mathscr{K}(T),\quad T \gg 1,
\end{equation}
where $c_{25} > 0$ is given in Lemma \ref{lemma:9}. It is evident that $\mathscr{K}(T)$ is uniquely determined and that the hypotheses of Lemma \ref{lemma:9} are satisfied. Observe that since $\Lambda(x) \gtrsim x$ and $\Psi(x) \gg 1$ hold for sufficiently large $x$, we have
\[\Gamma(x) = \Psi(x)\log \Lambda(x) \geqslant \beta\log x.\]
Consequently, by \eqref{ddgx}, there exists a constant $c_{30} > \beta c_{25} + 1$ such that for sufficiently large $T$,
\[T = \mathscr{K}^\beta(T)\exp\big(\beta c_{25}\Gamma(\mathscr{K}(T))\big) = \exp\big(\beta \log \mathscr{K}(T) + \beta c_{25}\Gamma(\mathscr{K}(T))\big) \leqslant \exp\big(c_{30}\Gamma(\mathscr{K}(T))\big),\]
which implies that
\begin{equation}\label{KGJ}
	\mathscr{K}(T) \geqslant \Gamma^{-1}(c_{30}^{-1}\log T).
\end{equation}
At this stage, by invoking Lemmas \ref{lemma:9} and \ref{lemma:10} alongside \eqref{ddgx} and \eqref{KGJ}, we immediately obtain that, as $T \to +\infty$,
\begin{align*}
	\sum_{k \in \widetilde{\mathcal{S}}_{\mathrm{T}}} a_k D^j\varphi(x(T) + \Omega_k T) &= \mathcal{O}\Big(\exp\big(-c_{24}T^{\frac{1}{\beta}}\exp(-c_{25}\Gamma(\mathscr{K}(T)))\big)\Big) \\
	&= \mathcal{O}\Big(\exp\big(-c_{24}\mathscr{K}(T)\big)\Big) \\
	&= \mathcal{O}\Big(\exp\big(-c_{24}\Gamma^{-1}(c_{30}^{-1}\log T)\big)\Big), \quad 0 \leqslant j \leqslant 1, 
\end{align*}
and
\begin{align*}
	\sum_{k \in \widetilde{\mathcal{S}}_{\mathrm{R}}} a_k D^j\varphi(x(T) + \Omega_k T) &= \mathcal{O}\Big(\exp\big(-c_{29}\mathscr{K}(T)\big)\Big) \\
	&= \mathcal{O}\Big(\exp\big(-c_{29}\Gamma^{-1}(c_{30}^{-1}\log T)\big)\Big), \quad 0 \leqslant j \leqslant 1.
\end{align*}
Finally, substituting these estimates back into \eqref{fjzh}, we deduce that as $T \to +\infty$,
\[\sum_{k \in \Upsilon} a_k D^j\varphi(x(T) + \Omega_k T) = \mathcal{O}\Big(\exp\big(-c_{31}\Gamma^{-1}(c_{30}^{-1}\log T)\big)\Big), \quad 0 \leqslant j \leqslant 1,\]
where $ {c_{31}} = \min \left\{ {{c_{24}},{c_{29}}} \right\}>0 $. Lastly, following an analysis similar to that in \eqref{xT} and \eqref{bjgj}, we conclude that as $T \to +\infty$,
\begin{align}
\nu_1^T - \nu_1 &= \mathcal{O}(T^{-1}x(T)) \notag  \\
&= \mathcal{O}\Big(T^{-1}\exp\big(-c_{31}\Gamma^{-1}(c_{30}^{-1}\log T)\big)\Big)\notag\\ 
&= \mathcal{O}\Big(\exp\big(-c_{\rm IV}\Gamma^{-1}(c_{\rm V}\log T)\big)\Big),\label{zzgj2}
\end{align}
where $c_{\rm IV} :=c_{31}>0$, and $c_{\rm V} := c_{30}^{-1} > 0$. This completes the proof of the abstract part of Theorem \ref{TH4}.

In what follows, we proceed to establish the corollaries of Theorem \ref{TH4} (i.e., Cases \ref{item1}, \ref{item2} and \ref{item3}) in sequence. 

Assuming that $\Phi(x)\sim\log x$ as $x\to+\infty$, we deduce that $\Lambda(x)=\Phi^{-1}(x)\asymp e^x$, where $\asymp$ denotes the same order of magnitude. Moreover, evaluating the asymptotic behavior of the integral yields
\[\int_0^x\Phi(y)dy\sim\int_0^x\log ydy\sim x\log x,\]
which implies that $\Psi(x)\sim x(\log x)^{-1}$. It then follows that
\[\Gamma(x)=\Psi(x)\log\Lambda(x)\sim x^2(\log x)^{-1},\]
and thus the inverse function satisfies $\Gamma^{-1}(x)\sim\sqrt{2^{-1}x\log x}$. Hence, an application of \eqref{zzgj2} gives
\begin{align*}
	\nu_1^T-\nu_1&=\mathcal{O}\Big(\exp\big(-c_{\mathrm{IV}}^*\sqrt{2^{-1}c_{\mathrm{V}}\log T\log(c_{\mathrm{V}}\log T)}\big)\Big)\\
	&=\mathcal{O}\Big(\exp\big(-c_{\mathrm{VI}}\sqrt{\log T\log\log T}\big)\Big),\quad T\to+\infty,
\end{align*}
provided that the constants satisfy $0<c_{\mathrm{IV}}^*<c_{\mathrm{IV}}$ and $0<c_{\mathrm{VI}}<2^{-1}c_{\mathrm{IV}}^*c_{\mathrm{V}}^{1/2}$. This completes the proof of Case \ref{item1}.

Suppose alternatively that $\Phi(x)\sim e^x$ as $x\to+\infty$. In this case, we have $\Lambda(x)=x$. Turning to the integral, its asymptotic behavior can be evaluated as
\[\int_0^x\Phi(y)dy\sim\int_0^x e^y dy\sim e^x,\]
which gives rise to the relation $\Psi(x)\sim\log x$. Consequently, we obtain
\[\Gamma(x)=\Psi(x)\log\Lambda(x)\sim(\log x)^2,\]
yielding the asymptotic bound for the inverse function $\Gamma^{-1}(x) \asymp \exp(\sqrt{x})$. Therefore, invoking \eqref{zzgj2} leads to
\begin{align*}
	\nu_1^T-\nu_1&=\mathcal{O}\Big(\exp\big(-c_{\mathrm{IV}}^{**}\exp(\sqrt{c_{\mathrm{V}}\log T})\big)\Big)\\
	&=\mathcal{O}\Big(\exp\big(-\exp(c_{\mathrm{VII}}\sqrt{\log T})\big)\Big),\quad T\to+\infty,
\end{align*}
provided that the appropriate constants satisfy $0<c_{\mathrm{IV}}^{**}<c_{\mathrm{IV}}$ and $0<c_{\mathrm{VII}}<c_{\mathrm{V}}^{1/2}$. This completes the proof of Case \ref{item2}.

To establish Case \ref{item3}, we bypass the estimate in \eqref{zzgj2} and instead rely directly on the estimates provided in Lemmas \ref{lemma:9} and \ref{lemma:10}. For this purpose, we assume that $\Phi(x) \geqslant x$, hence $\Lambda(x) = x$. Observe that for any given function $\mathscr{I}: [1,+\infty) \to [1,+\infty)$ satisfying $\mathscr{I}(x) = o(x^\beta)$ as $x \to +\infty$, one can always construct a sufficiently large truncation function $\mathscr{K}(x)$ such that $\mathscr{K}(T) \geqslant c_{29}^{-1} \mathscr{I}(T)$. Furthermore, once $\mathscr{K}(x)$ is fixed, we can always choose a sufficiently small function $\Gamma(x)$ satisfying
\[T^{\frac{1}{\beta}}\exp\big(-c_{25}\Gamma(\mathscr{K}(T))\big) \geqslant c_{24}^{-1} \mathscr{I}(T),\quad T\gg 1.\]
Consequently, combining the estimates from Lemmas \ref{lemma:9} and \ref{lemma:10} with our subsequent analysis, we deduce that
\[\nu_1^T - \nu_1 = \mathcal{O}\Big(\exp\big(-\mathscr{I}(T)\big)\Big),\quad T \to +\infty.\]
Since $\Gamma(x) = \Psi(x)\log \Lambda(x) = \Psi(x)\log x$, ensuring that $\Gamma(x)$ is appropriately small simply requires $\Psi(x)$ to be sufficiently small. By definition, this can be achieved by making the integral $\int_0^x \Phi(y)dy$ sufficiently large, which in turn guarantees the existence of such a function $\Phi(x)$. This completes the proof of Case \ref{item3}.

The proof of Theorem \ref{TH4} is now complete.

 \section*{Acknowledgements} 
 Z. Tong was supported by the China Postdoctoral Science Foundation (Grant No. 2025M783102). Y. Li was supported in part by the National Natural Science Foundation of China (Grant Nos. 12471183 and 12531009).

\end{document}